\documentclass[11pt,reqno]{amsart}
\usepackage{a4wide}
\usepackage{amssymb}
\usepackage{amsthm}
\usepackage{amsmath,todonotes}
\usepackage{amscd}
\usepackage{cancel}
\usepackage{verbatim}
\usepackage{color}
\usepackage[all]{xy}
\usepackage[mathscr]{eucal}
\usepackage{mathrsfs}
\usepackage{fullpage}
\usepackage{enumitem}
\usepackage{hyperref}
\usepackage{bbm}
\usepackage{qtree}
\usepackage{bm}
\allowdisplaybreaks

\numberwithin{equation}{section}

\usepackage[headheight=110pt,top=1in, bottom=1in, left=0.8in, right=0.8in]{geometry}

\newtheorem{theorem}{Theorem}

\newtheorem{lemma}[theorem]{Lemma}

\newtheorem{corollary}[theorem]{Corollary}
\newtheorem*{conjecture}{\bf Conjecture}

\newtheorem{proposition}[theorem]{Proposition}

\theoremstyle{remark}

\theoremstyle{definition}

\numberwithin{theorem}{section} 
\numberwithin{equation}{section}
\numberwithin{table}{section}

\newcommand{\ord}{\text {\rm ord}}

\newcommand{\GG}{\mathcal{G}}

\newcommand{\R}{\mathbb{R}}

\newcommand{\Q}{\mathbb{Q}}

\newcommand{\Z}{\mathbb{Z}}
\newcommand{\N}{\mathbb{N}}
\newcommand{\SL}{{\text {\rm SL}}}

\newcommand{\lcm}{{\text {\rm lcm}}}

\def\H{\mathbb{H}}
\renewcommand{\(}{\left(}
\renewcommand{\)}{\right)}

\newcommand{\vol}{\operatorname{vol}}
\newcommand{\Mod}{ \, (\operatorname{mod} \,}
\newcommand{\x}{\bm{\mathrm{x}}}
\newcommand{\spn}{{\rm spn} \,}
\newcommand{\gen}{{\rm gen} \,}

\begin{document}
\title[On Sun's conjecture]{On a Conjecture of Sun about sums of restricted squares}
\author{Soumyarup Banerjee}
\address{Discipline of Mathematics, Indian Institute of Technology Gandhinagar, Palaj, Gandhinagar-382355, Gujarat, India.}
\email{soumyarup.b@iitgn.ac.in}

\medskip
\begin{abstract}
In this paper, we investigate sums of four squares of integers whose prime factorizations are restricted, making progress towards a conjecture of Sun that states that two of the integers may be restricted to the forms $2^a3^b$ and $2^c5^d$. We obtain an ineffective generalization of results of Gauss and Legendre on sums of three squares and an effective generalization of Lagrange's four-square theorem.
\end{abstract}

\maketitle

\section{Introduction And Statement Of Results}\label{sec:intro}
The study of representations of integers by sums of integral squares goes back to antiquity and has a storied history. To give one famous example, Legendre (in 1797) and Gauss (in 1796--1801) separately proved that every natural number not of the form $4^q(8\ell +7)$ can be represented as the sum of three squares of non-negative integers. Moreover, Gauss' work culminated in a formula that relates the number of representations $r_3(m)$ of $m$ as a sum of three squares to a class number of an associated imaginary quadratic field. Letting $H(D)$ denote the Hurwitz class number, Gauss' result may be stated as 
\begin{equation*}
r_3(m) = \begin{cases}
12 H(-4m) & \text{ if } n \equiv 1, 2 \pmod 4,\\
24 H(-m) &  \text{ if } n \equiv 3 \pmod 8,\\
r_3(\frac{m}{4}) &  \text{ if } n \equiv 0 \pmod 4,\\
0 &  \text{ if } n \equiv 7 \pmod 8.
\end{cases}
\end{equation*}
Prior to Gauss and Legendre's works on sums of three squares, Lagrange established in 1770 that every natural number can be represented as the sum of four squares of non-negative integers. 
Jacobi later found a formula in 1834 analogous to that of Gauss for the number of representation $r_4(m)$ of $m$ as a sum of four square, yielding 
\begin{equation}\label{Jacobi}
r_4(m) = 8\sum_{\substack{d\mid m\\ 4\nmid d}} d.
\end{equation}
Lagrange's four-square theorem and Jacobi's formula \eqref{Jacobi} have been generalized in numerous directions through the years, with some results extending the types of sums being taken and others restricting the integers being squared into certain subsets. Along this vein, Sun recently stated a {\em four square conjecture} where some of the integers are restricted to be products of powers of $2$, $3$, and $5$, as stated below.
\begin{conjecture}
Every $n = 2, 3, \cdots$ can be written as $x^2 + y^2 +(2^a 3^b)^2 + (2^c 5^d)^2$, where $x, y, a, b, c, d$ are non-negative integers.
\end{conjecture}
This conjecture seems to be out of reach with current techniques, but, following results of Br\"udern--Fouvry \cite{BrudernFouvry}, some progress can be made in restricting the number of prime divisors of the last two squares. The goal of this article is to demonstrate how to use such techniques to generalize both Gauss/Legendre's three-square theorem and Lagrange's four-square theorem in the direction of Sun's four square conjecture. The following result generalizes Gauss' three square theorem ineffectively in the sense that every sufficiently large integer not of the form $4^q(8\ell +7)$ can be represented by sum of three squares where the last variable can be restricted to almost prime inputs. Here an \begin{it}almost prime\end{it} of order $n$ is a product of at most $n$ primes.
\begin{theorem}\label{Thm 1}
Every sufficiently large integer $m$ not of the form $4^q(8\ell +7)$ can be represented in the form 
$$m = x^2 + y^2 + (2^a z)^2,$$
where $x$, $y$, $a$ and $z$ are any integers with $a$ non-negative and $z$ has at most $118$  prime factors. Moreover, the number of such representation exceeds $cm^{1/2-\epsilon}(\log m)^{-1}$ for some positive constant $c$.  
\end{theorem}
The next result provides a generalization of Lagrange's four square theorem ineffectively in the sense that every sufficiently large integer can be represented by sum of four squares with restricted inputs.
\begin{corollary}\label{Cor 1.2}
Every sufficiently large integer $m$ can be represented in the form
\begin{equation*}
m = x^2 + y^2 + 2^{2a} + (2^bz)^2
\end{equation*}
where $x$, $y$, $a$, $b$ and $z$ are any integers with $a$, $b$ non-negative and $z$ has at most $118$  prime factors.
\end{corollary}
The proof of the above corollary follows from Theorem \ref{Thm 1} by a simple argument. Namely, writing $m = 4^km'$ with $4\nmid m'$, we can choose $a$ such that $m'-4^a\not\equiv 7\pmod{8}$. Hence Corollary \ref{Cor 1.2} follows from Theorem \ref{Thm 1}.

In connection to Sun's four square conjecture, the following result holds directly from Theorem \ref{Thm 1}.
\begin{corollary}
Every sufficiently large integer $m$ can be represented in any of the following form :
$$m = x^2 + y^2 + (2^a3^b z)^2 + (2^c 5^d)^2,$$ 
or
$$m = x^2 + y^2 + (2^a3^b)^2 + (2^c 5^d z)^2,$$
where $x, y, a, b, c, d$ are non-negative integers and $z$ has at most $118$ prime factors. 
\end{corollary}
Combining Gauss's result with Siegel's lower bound for the class numbers \cite{SiegelClassNumber} yields that for arbitrarily small $\epsilon_1 > 0$
\begin{equation}\label{Siegel's bound}
r_3(m)\gg h(-m) \gg m^{1/2 - \epsilon_1},
\end{equation}
where $h(-m)$ denotes the class number of the number field $\Q(\sqrt{-m})$. However, Siegel's lower bound is ineffective, so the bound on $m$ for which Theorem \ref{Thm 1} is not effective. Under the assumption of the generalized Riemann hypothesis, Siegel's result can be made effective and Ono and Soundarajan \cite{Ono} worked out an explicit bound in order to obtain an conjectural proof of a conjecture of Ramanujan about sums of the form $x^2+y^2+10z^2$. Following this method, an effective but conjectural version of Theorem \ref{Thm 1} can be obtained. 
    
By further relaxing the conditions on the last two integers being squared, we obtain an effective unconditional version of Theorem \ref{Thm 1}. Moreover, using a quantitative version \cite{BK} of results of Br\"udern--Fouvry \cite{BrudernFouvry} (see also \cite{Tsang} for th current state of the art), one can make this effective constant explicit, leading to the conclusion that indeed every integer may be written in a certain shape.
\begin{theorem}\label{Thm 2}
Every natural number $m$ can be represented in the form of 
$$m = x^2 + y^2 + (2^a3^bz_1)^2 + (2^c 5^d z_2)^2,$$ 
where $x, y, a, b, c, d$ are non-negative integers and $z_1, z_2$ each either vanish or have at most $369$ prime factors.
\end{theorem}

The paper is organized as follows. In \S \ref{S2}, we introduce the preliminaries needed for the rest of the paper. In \S \ref{S3}, we prove bounds on the coefficients of theta functions. We apply a linear sieve to prove Theorem \ref{Thm 1} in \S \ref{S4}. In \S \ref{S5}, we obtain the bounds required to prove Theorem \ref{Thm 2}. Finally, Theorem \ref{Thm 2} is proved in \S \ref{S6}.
\section{Preliminaries}\label{S2}
Let $Q$ be any $\ell$-ary positive definite integer valued diagonal quadratic form with $\ell \geq 3$ and $r_Q(m)$ denotes the number of solutions to $Q(\bm{x})=m$ for $\bm{x}\in \Z^{\ell}$. 

\subsection{Theta function}
Let $\H$ be the complex upper half-plane. For $\tau\in \H$ and $q = e^{2\pi i \tau}$, the theta function associated to the quadratic form $Q$ can be defined as
\[
\Theta_Q(\tau):=\sum_{m\geq 0} r_Q(m) q^m = \sum_{\bm{x}\in \Z^{\ell}}q^{Q(\bm{x})}.
\]
The theta function $\Theta_Q$ is a modular form of weight $\ell/2$ on a particular congruence subgroup $\Gamma$ of $\SL_2(\Z)$ with a certain Nebentypus $\chi$ (cf. \cite[Proposition 2.1]{Shimura}). It naturally decomposes as
\[
\Theta_Q=E+f,
\]
where $E$ is an Eisenstein series of weight $\ell/2$ on $\Gamma$ with Nebentypus $\chi$ and $f$ is a cusp form of same weight $\ell/2$ on $\Gamma$ with Nebentypus $\chi$. It turns out that the $m$th coefficient $a_E(m)$ of $E$ grows faster than the $m$th coefficient $a_{f}(m)$ of $f$ for $\ell\geq 3$ and for those $m$ for which $a_{E}(m)$ is non-zero. Moreover, it can be shown that $a_{E}(m)>0$ if and only if $m$ is represented locally i,e. represented modulo any natural number. Therefore, $E$ behaves as the main term of $\Theta_Q$ and $f$ as the error term, with $r_Q(m)>0$ for sufficiently large $m$ that are locally represented. 

\subsection{Eisenstein series part}
Siegel mainly used the ideas to give quantitative meaning to the Fourier coefficients of Eisenstein series associated to the quadratic form $Q$  both in terms of the underlying space of modular forms and also in terms of {\it local densities} associated to $Q$. The Eisenstein series
$$E(\tau) := \sum_{m \geq 0} a_E(m)q^m$$
can be expressed in two different ways :

Firstly, for $\GG(Q)$ denoting a set of representatives of the classes in the genus of $Q$ and $w_Q$ denoting the number of automorphs of $Q$, the Eisenstein series $E$ can be recovered as a weighted sum of theta series over $\GG(Q)$ by  
\begin{equation}\label{eqn:SiegelWeil}
E = \frac{1}{\sum_{Q'\in\GG(Q)}w_{Q'}^{-1}}  \sum_{Q'\in\GG(Q)} \frac{\Theta_{Q'}}{w_{Q'}}. 
\end{equation}
This famous identity is known as {\it Siegel -- Weil average} (the first identity is due to Siegel \cite{Siegel} and a generalization by Weil \cite{Weil}).

We recall here local representation densities for an $\ell$-ary quadratic form $Q$  at $m$ which can be defined by the limit
\begin{equation}\label{Def local density}
\beta_{Q, p}(m):=\lim_{U\to\{m\}} \frac{\vol_{\Z_p^{\ell}} \left(Q^{-1}(U)\right)}{\vol_{\Z_p}(U)}, 
\end{equation}
where $U\subseteq\Z_p$ runs over open subsets of $\Z_p$ containing $m$ and for $p=\infty$ we have open subsets of $\R$. The Fourier coefficients $a_E(m)$ of $E$ can be expressed as an infinite local product
\begin{equation}\label{Local density}
a_E(m) = \prod_{p} \beta_{Q, p}(m),
\end{equation}
where the product runs over all the primes including $\infty$.

\section{Representation of sufficiently large integer by certain ternary quadratic forms}\label{S3}
In this section, we will set up the background and notations for Theorem \ref{Thm 1}. The set to be sieved is the following 
 $$\mathcal{A} = \{z \in \N : x^2 + y^2 + (2^a z)^2= m \},$$
 where $x, y, a$ are any non-negative integers. Setting  $x_1 =x$, $x_2 = y$ and $x_3 = 2^a z$, we have $$\mathcal{A} = \{x_3 \in \N : x_1^2 + x_2^2 + x_3^2= m \}.$$
In order to apply sieve theory, we need an asymptotic formula for the cardinality of the  set
$$\mathcal{A}_{d} = \{x_3 \in \mathcal{A} : x_3 \equiv 0 \Mod d) \} = \{x_3 \in \N :  x_1^2 + x_2^2 + d^2x_3^2= m \},$$
where $2 \nmid d$. It needs some preparation to express the above cardinality in terms of main term and error term. We first consider the quadratic form
$$Q_{d}(\bm{\mathrm{x}}) = x_1^2 + x_2^2 + d^2 x_3^2,$$ 
where $\bm{\mathrm{x}} = (x_1, x_2, x_3)$ and $2 \nmid d$. For simplicity, we abbreviate $Q_1 = Q$. Let, the theta function associated to $Q_{d}$ be $\Theta_{Q_{d}}$ with the Fourier expansion 
$$\Theta_{Q_{d}}(\tau) = \sum_{m\geq 0} r_{Q_{d}}(m) q^n$$
for $\tau \in \H$. Here the Fourier coefficient $r_{Q_{d}}(m)$ denotes the number of representation of $m$ by  $Q_{d}$. As mentioned earlier, $\Theta_{Q_{d}}$ can be decomposed into two parts which are the Eisenstein series part and the cuspidal part respectively.  

\subsection{Eisenstein series contribution}
We denote the Eisenstein series part associated to $Q_{d}$ by $E_{d}$. It follows from \eqref{Local density} that the $m$-th Fourier coefficient of $E_{d}$ can be expressed as 
\begin{equation}\label{Eisenstein product}
a_{E_{d}}(m) = \prod_{p} \beta_{Q_{d,} p}(m),
\end{equation}  
where the product runs over all the primes including $\infty$. The definition \eqref{Def local density} of local representation density yields 
\[
\beta_{Q_{d}, p}(m)=\lim_{U\to\{m\}} \frac{\vol_{\Z_p^{3}} \left(Q^{\leftarrow}(U)\right)}{\vol_{\Z_p}(U)}, 
\]
where $U\subseteq\Z_p$ runs over open subsets of $\Z_p$ containing $m$ and for $p=\infty$ we have open subsets of $\R$. For $p \neq \infty$, the local representation density may be realized by choosing $U$ to be a ball of radius $p^{-r}$ around $m$, in which case we may write
\begin{equation}\label{Local density p}
\beta_{Q_{d,} p}(m) = \lim_{r \to \infty}\frac{|R_{Q_{d,} p^r}(m)|}{p^{2r}}
\end{equation}
with $$R_{Q_{d,} p^r}(m) := \{ \x \in (\Z/p^r \Z)^3 : Q_{d}(\x) \equiv m \Mod p^r) \}.$$ The following lemma is crucial to determine the local density of $Q_{d}$ at $p =2$. 
\begin{lemma}\label{Lemma for p not divide d}
For $p \nmid d$, the local density of $Q_{d}$ satisfies
$$\beta_{Q_{d,} p} = \beta_{Q, p}.$$
In particular, we have $$\beta_{Q_{d,} 2} = \beta_{Q, 2}.$$
\end{lemma}
\begin{proof}
The proof of the lemma is straight forward. For $p\nmid d$, there exist $d^{-1} \in \Z$ such that $dd^{-1} \equiv 1 \Mod p) $. It follows that there exist a bijection between the sets $R_{Q, p^r}(m)$ and $R_{Q_{d,} p^r}(m)$ under the map $(x_1, x_2, x_3) \mapsto (x_1, x_2, d^{-1}x_3)$. This implies $|R_{Q, p^r}(m)| = |R_{Q_{d,} p^r}(m)|$. We now use  \eqref{Local density p}, which yields the lemma. The lemma follows in particular for $p=2$ since by assumption we have considered $d$ to be odd.
\end{proof}
We fix some notations here before proceeding to the next lemma.
Set for any $d$ odd,
\[
\varepsilon_d:=\begin{cases} 1 &\text{for }d\equiv 1\pmod{4},\\ i&\text{for }d\equiv 3\pmod{4}.\end{cases}
\]
Let the symbols $\left(\frac{\cdot}{\cdot}\right)$ and $[ \cdot ]$ denote the Legendre--Jacobi--Kronecker symbol and the greatest integer function respectively. We set $d = p^{\alpha}d'$ with $p\nmid d'$ and $m=p^Rm'$ with $p\nmid m'$. Let $\delta_\ell$ denotes the standard characteristic function which counts $1$ if $\ell$ happens and vanishes otherwise. In the next lemma, we compute the local representation density $\beta_{Q_{d,} p}$ at each odd prime $p$.
\begin{lemma}{\label{Lemma for local computation}}
For any odd prime $p$, we have
\begin{equation*}
\beta_{Q_{d}, p} (m)= 
\begin{cases}
1 + (1 - p^{-1})\left( \left[ \frac{R}{2} \right] + \varepsilon_p^2 \left[\frac{R+1}{2}\right]\right) - \delta_{2\nmid R} \, p^{-1} - \delta_{2\mid R} \, \varepsilon_p^2 \, p^{-1}
& \text{for  } R < 2 \alpha, \vspace{.5cm}\\
\begin{aligned}
1 + (1 + \varepsilon_p^2)&(1-p^{-1})\alpha + p^{-1}
 - p^{ \alpha-1 -  [\frac{R}{2}]}\\
&- \delta_{2\nmid R} \, p^{\alpha -1 - \frac{R+1}{2}}
+ \delta_{2\mid R} \, \varepsilon_p^2 \, p^{\alpha -1 - \frac{R}{2}} \left(\frac{m'}{p}\right) 
\end{aligned}
&  \text{for  } R \geq 2 \alpha.
\end{cases}
\end{equation*}
\end{lemma}
\begin{proof}
The orthogonality of roots of unity, namely 
\begin{equation}\label{Orthogonality of roots of unity}
\frac{1}{p^r}  \sum_{n \Mod p^r)} e^{\frac{2\pi i nm}{p^r}} = \begin{cases} 1 &\text{if }p^{r}\mid m,\\ 0 &\text{otherwise},\end{cases}
\end{equation}
leads to 
\begin{align*}
|R_{Q_{d,} p^r}(m)|&=\sum_{\substack { \x \in(\Z/p^r\Z)^3\\ Q_{d}(\bm{\widetilde{x}}) \, \equiv \, m \Mod p^r)}}1\\
& = \sum_{\x \in(\Z/p^r\Z)^3}\frac{1}{p^r} \sum_{n \Mod p^r)} e^{\frac{2\pi i n}{p^r}\left(Q_{d}(\x)-m\right)}\\
&= \frac{1}{p^r} \sum_{n\Mod p^r)} e^{-\frac{2\pi i nm}{p^r}} \sum_{\x\in(\Z/p^r\Z)^3}e^{\frac{2\pi i n}{p^r}Q_{d}(\x)}\\
&= \frac{1}{p^r} \sum_{n\Mod p^r)} e^{-\frac{2\pi i nm}{p^r}}\(\prod_{j=1}^2 \sum_{x_j\in \Z/p^r\Z}e^{\frac{2\pi i nx_j^2}{p^r}} \) \sum_{x_3 \in \Z/p^r\Z}e^{\frac{2\pi i nd^2x_3^2}{p^r}}\\
&= \frac{1}{p^r} \sum_{n\Mod p^r)} e^{-\frac{2\pi i nm}{p^r}}G_2(n, 0, p^r)^2G_2(nd^2, 0, p^r),
\end{align*}
where in the last step we used the definition of the quadratic Gauss sum, which is given by
\[
G_2(A,B,C):=\sum_{x\, (\operatorname{mod}\, c)} e^{\frac{2\pi i\left(Ax^2+Bx\right)}{C}}.
\]
We next split the sum over $n$ by writing $n=p^kn'$ with $p\nmid n'$ and then make the change of variables $k\mapsto r-k$. The fact that 
\begin{equation}\label{eqn:G2gcd}
G_2(gA,gB,gC)=gG_{2}(A,B,C)
\end{equation}
 yields
\begin{align*}
|R_{Q_{d,} p^r}(m)| &= \frac{1}{p^r} \sum_{k=0}^r p^{3k}\sum_{n'\in (\Z/p^{r-k}\Z)^{\times}} e^{-\frac{2\pi i n'm}{p^{r-k}}}G_2(n', 0, p^{r-k})^2G_2(n'd^2, 0, p^{r-k})\\
&= p^{2r} \sum_{k=0}^r p^{-3k}\sum_{n'\in (\Z/p^{k}\Z)^{\times}} e^{-\frac{2\pi i n'm}{p^{k}}}G_2(n', 0, p^{k})^2G_2(n'd^2, 0, p^{k}).
\end{align*}
For $d = p^{\alpha}d'$ with $p\nmid d'$, it follows from \eqref{eqn:G2gcd} that
\begin{align}\label{G_2 evaluation}
G_2(n'd^2, 0, p^{k}) &= (p^{2\alpha}, p^k) \, G_2\(\frac{n'd^2}{(p^{2\alpha}, p^k)}, 0, \frac{p^k}{(p^{2\alpha}, p^k)}\)\nonumber\\
&= \begin{cases}
p^k & \text{for } k\leq 2\alpha, \\
p^{2\alpha}G_2(n'd'^2, 0, p^{k-2\alpha}) & \text{for } k>2\alpha.
\end{cases}
\end{align}
We can therefore rewrite
\begin{align}\label{Eisenstein computation 1}
\frac{1}{p^{2r}}|R_{Q_{d,} p^r}(m)| &= \sum_{k=0}^{2\alpha} p^{-2k}\sum_{n'\in (\Z/p^{k}\Z)^{\times}} e^{-\frac{2\pi i n'm}{p^{k}}}G_2(n', 0, p^{k})^2 \nonumber\\
&+ \sum_{k=2\alpha+1}^{r} p^{-3k + 2\alpha} \sum_{n'\in (\Z/p^{k}\Z)^{\times}} e^{-\frac{2\pi i n'm}{p^{k}}}G_2(n', 0, p^{k})^2G_2(n'd'^2, 0, p^{k-2\alpha}).
\end{align}
Utilizing the fact that for $C$ odd and for $\gcd(A,C)=1$,
\[
G_2(A,0,C)=\varepsilon_C\sqrt{C}\left(\frac{A}{C}\right),
\]
we have, for $p\neq 2$, 
\begin{equation}\label{eqn:G2(1)}
G_2(n', 0, p^k)=\begin{cases} 
p^{\frac{k}{2}} &\text{if }k\equiv 0 \Mod{2}),\\ 
\varepsilon_p\left(\frac{n'}{p}\right) p^{\frac{k}{2}}&\text{if }k\equiv 1 \Mod{2}).
\end{cases}
\end{equation}
and
\begin{align}\label{eqn:G2(2)}
G_2(n'd'^2, 0, p^{k-2\alpha}) = G_2(n', 0, p^{k-2\alpha}) = \begin{cases} 
p^{\frac{k}{2} - \alpha} &\text{if }k\equiv 0 \Mod{2}),\\ 
\varepsilon_p\left(\frac{n'}{p}\right) p^{\frac{k}{2} - \alpha}&\text{if }k\equiv 1 \Mod{2}).
\end{cases}
\end{align}
For a multiplicative character $\chi$ and an additive character $\psi$, both of modulus $c$, we set
\[
\tau(\chi,\psi):=\sum_{x\Mod{c})} \chi(x)\psi(x).
\]
Let $\chi=\chi_{a,b}$ denotes a character of modulus $b$ induced from a character of conductor $a$. For simplicity, we abbreviate $\chi_{a,a} = \chi_a$ (we will always have either the principal character $\chi_{1,p^k}$ or the real Dirichlet character $\chi_{p,p^k}=\left(\frac{\cdot}{p}\right)$ coming from the Legendre symbol) and take $\psi(x)=\psi_{m, p^k}(x):=e^{\frac{2\pi i mx}{p^k}}$. Inserting \eqref{eqn:G2(1)} and \eqref{eqn:G2(2)} into \eqref{Eisenstein computation 1}, we obtain 
\begin{multline}\label{Eisenstein computation 2}
\frac{1}{p^{2r}}|R_{Q_{d,} p^r}(m)| = \sum_{k=0}^{2\alpha}(\delta_{2\mid k} + \delta_{2\nmid k}\varepsilon_p^2)\, p^{-k} \tau(\chi_{1, p^k}, \psi_{-m, p^k}) \\
+ \sum_{\substack{k=2\alpha+1 \\ k \text{ even}}}^{r} p^{\frac{-3k}{2}+\alpha} \tau(\chi_{1, p^k}, \psi_{-m, p^k}) 
+ \varepsilon_p^3 \sum_{\substack{k=2\alpha+1 \\ k \text{ odd}}}^{r} p^{\frac{-3k}{2}+\alpha} \tau(\chi_{p, p^k}, \psi_{-m, p^k}).
\end{multline}
We next evaluate $\tau(\chi,\psi)$. Letting $\chi^*$ denote the primitive character of modulus $m^*$ associated to the character $\chi$ of modulus $m$ and abbreviating $\tau(\chi^*):=\tau(\chi^*,\psi_{1,m^*})$, a corrected version of \cite[Lemma 3.2]{IwaniecKowalski} yields
\[
\tau(\chi,\psi_{a,m})=\tau(\chi^*)\sum_{d \mid \gcd \left(a,\frac{m}{m^*}\right)}d \chi^*\left(\frac{m}{m^*d}\right)\overline{\chi^*\left(\frac{a}{d}\right)} \mu\left(\frac{m}{m^*d}\right).
\]
Hence we have (noting that $\left(\frac{n}{p} \right)=0$ if $p\mid n$) 
\begin{align}\label{eqn:tau1} 
\tau\left(\chi_{1,p^k},\psi_{-m,p^k}\right)&=\sum_{d\mid \gcd(m,p^k)} d \mu\left(\frac{p^{k}}{d}\right)=\begin{cases} 1 &\text{if }k=0,\\ -p^{k-1}&\text{if }\gcd(m, p^k)=p^{k-1},\\ p^{k}-p^{k-1}&\text{if }\gcd(m, p^{k})=p^k,\\ 0&\text{otherwise},\end{cases}
\end{align}
\begin{align}\label{eqn:taup} 
 \tau\left(\chi_{p,p^k},\psi_{-m,p^k}\right)&=\tau(\chi_p)\sum_{d\mid \gcd(m,p^{k-1})} d \chi_p\left(\frac{p^{k-1}}{d}\right)\chi_p\left(-\frac{m}{d}\right)\mu\left(\frac{p^{k-1}}{d}\right)\nonumber\\
&=\begin{cases} p^{k-1}\tau(\chi_p)\chi_p\left(-\frac{m}{p^{k-1}}\right) &\text{if }\ord_p(m)=k-1\\
0&\text{otherwise}.
\end{cases}\nonumber\\
&=\begin{cases} \varepsilon_p^3 p^{k-\frac{1}{2}}\chi_p\left(\frac{m}{p^{k-1}}\right) &\text{if }\ord_p(m)=k-1\\
0&\text{otherwise},
\end{cases}
\end{align}
where in the last step we have used Gauss's evaluation $\tau(\chi_p)=\varepsilon_p \sqrt{p}$. Now for $m=p^Rm'$ with $p\nmid m'$, it follows from \eqref{eqn:tau1} and \eqref{eqn:taup} that each term with $k>R+1$ in \eqref{Eisenstein computation 2} vanishes. If $R < 2\alpha$, we use \eqref{eqn:tau1} to obtain
\begin{align*}
\lim_{r \to \infty} \frac{1}{p^{2r}}|R_{Q_{d,} p^r}(m)| &= 1 + \sum_{k=1}^R (\delta_{2\mid k} + \delta_{2\nmid k} \, \varepsilon_p^2) p^{-k} (p^k - p^{k-1}) - \delta_{2\nmid R} \, p^{-1} - \delta_{2\mid R} \, \varepsilon_p^2 \, p^{-1}\\
&= 1 + (1 - p^{-1})\left( \left[ \frac{R}{2} \right] + \varepsilon_p^2 \left[\frac{R+1}{2}\right]\right) - \delta_{2\nmid R} \, p^{-1} - \delta_{2\mid R} \, \varepsilon_p^2 \, p^{-1}.
\end{align*}
Finally for $R \geq 2\alpha$, we insert \eqref{eqn:tau1} and \eqref{eqn:taup} into \eqref{Eisenstein computation 2} to conclude
\begin{align*}
\lim_{r \to \infty} \frac{1}{p^{2r}}|R_{Q_{d,} p^r}(m)| = 1 + (1 + \varepsilon_p^2)(1-p^{-1})\alpha &+ \sum_{\substack{k = 2\alpha + 1 \\ k \text{ even}}}^R p^{\frac{-3k}{2}+\alpha}(p^k - p^{k-1}) \\
&- \delta_{2\nmid R} \, p^{\alpha - \frac{R+3}{2}}
+ \delta_{2\mid R} \, \varepsilon_p^2 \, p^{\alpha - \frac{R}{2} -1} \chi_p(m')\\
=  1 + (1 + \varepsilon_p^2)(1-p^{-1})\alpha &+ p^{-1} - p^{\alpha - [\frac{R}{2}] -1}\\ 
&- \delta_{2\nmid R} \, p^{\alpha - \frac{R+3}{2}}
+ \delta_{2\mid R} \, \varepsilon_p^2 \, p^{\alpha - \frac{R}{2} -1} \chi_p(m')
\end{align*}
This completes the proof of the Lemma.
\end{proof} 
We next shift our attention in computing the local densities at $p = \infty$. In the following lemma, we relate the local densities of $Q_{d}$ and $Q$ at $p = \infty$.
\begin{lemma}\label{Lemma Density at infinity}
We have 
\begin{equation}\label{Density at infinity}
\beta_{Q_{d}, \infty} = \frac{1}{d}\beta_{Q, \infty}.
\end{equation}
\end{lemma}
\begin{proof}
We compute $\beta_{Q_{d}, \infty}$  using the open sets $U = U_\epsilon := (m - \epsilon, m + \epsilon)$. It follows from the definition \eqref{Def local density} that 
\begin{equation}\label{eqn 3.3}
\beta_{Q_{d}, \infty} = \lim_{\epsilon \to 0} \frac{\vol_{\R^3}(Q_{d}^{-1}(U_\epsilon))}{\vol_{\R}(U_{\epsilon})} =  \lim_{\epsilon \to 0} \frac{\vol(B_{Q_{d}, m + \epsilon}) - \vol(B_{Q_{d}, m - \epsilon})}{2\epsilon},
\end{equation}
where $B_{Q_{d}, \ell}$ denotes the set $\{\bm{\mathrm{x}}\in \R^3 : Q_{d}(\bm{\mathrm{x}}) \leq \ell \}$ with $\bm{\mathrm{x}} = (x_1, x_2, x_3)$. Now, it remains to compute the volumes to conclude the lemma. 
We have
$$\vol(B_{Q_{d}, \ell}) := \int_{-\sqrt{\ell}}^{\sqrt{\ell}}\int_{- \sqrt{\ell - x_1^2}}^{\sqrt{\ell - x_1^2}}\int_{-\frac{\sqrt{\ell - x_1^2 -x_2^2}}{d}}^{\frac{\sqrt{\ell - x_1^2 -x_2^2}}{d}}{\rm d}x_3{\rm d} x_2 {\rm d} x_1.$$
The change of variable $x_3 \mapsto \frac{x_3}{d}$ yields
\begin{equation}\label{eqn vol}
\vol(B_{Q_{d}, \ell}) :=\frac{1}{d} \int_{-\sqrt{\ell}}^{\sqrt{\ell}}\int_{- \sqrt{\ell - x_1^2}}^{\sqrt{\ell - x_1^2}}\int_{-\sqrt{\ell - x_1^2 -x_2^2}}^{\sqrt{\ell - x_1^2 -x_2^2}}{\rm d}x_3{\rm d} x_2 {\rm d} x_1 = \frac{1}{d}\vol(B_{Q, \ell}).
\end{equation}
Hence the lemma follows by plugging \eqref{eqn vol} with $\ell = m+\epsilon$ and $\ell = m-\epsilon$ simultaneously into \eqref{eqn 3.3}.
\end{proof}
\subsection{Main term computation}\label{Definition of omega}
We define the multiplicative function $\bm{\omega}(m, d)$ for each square-free $d$ with $2 \nmid d$ by the following
\begin{equation}\label{Omega definition}
\bm{\omega}(m, d) := \prod_{p\mid d} \bm{\omega}(m, p)
\end{equation}
where $$\bm{\omega}(m, p) = \frac{\beta_{Q_{d,} p}(m)}{\beta_{Q, p}(m)}.$$
We next compare the coefficients $a_{E_{d}}(m)$ and $a_{E}(m)$ by plugging \eqref{Omega definition} and \eqref{Density at infinity} into \eqref{Eisenstein product} to obtain the main term of $|\mathcal{A}_{d}|$. It follows from Lemma \ref{Lemma for p not divide d} that
\begin{equation}\label{Main term}
a_{E_{d}}(m) = \frac{a_{E_{d}}(m)}{a_{E}(m)} a_{E}(m)= \frac{ \bm{\omega}(m, d)}{d}r_3(m).
\end{equation}
We can now proceed to find explicit formulas for $\bm{\omega}(m, p)$, which can be evaluated from Lemma \ref{Lemma for local computation}. 
\begin{lemma}\label{Omega for p does not divide m}
For $p\nmid m$, we have 
\begin{equation*}
\bm{\omega}(m, p) = \begin{cases}
\frac{1-p^{-1}}{1+p^{-1}} & \text{if } p\equiv 1 \Mod 4), (\frac{m}{p}) = 1 \vspace{.2cm}\\
\frac{1+p^{-1}}{1-p^{-1}} & \text{if } p\equiv 3 \Mod 4), (\frac{m}{p}) = 1 \vspace{.2cm}\\
1 & \text{if } p\equiv 1 \Mod 4), (\frac{m}{p}) = -1 \vspace{.2cm}\\
1 & \text{if } p\equiv 3 \Mod 4), (\frac{m}{p}) = -1.
\end{cases}
\end{equation*}
\end{lemma}

\begin{lemma}\label{Omega for p divide m in even powers}
For $m = p^{2\theta} m'$ with $\theta \geq 1$ and $(m', p) = 1$, we have
\begin{equation*}
\bm{\omega}(m, p) = \begin{cases}
\frac{3-p^{-1}}{1+p^{-1}} & \text{if } p\equiv 1 \Mod 4), (\frac{m'}{p}) = 1 \vspace{.2cm}\\
\frac{1+p^{-1} - 2p^{-\theta}}{1+p^{-1} - 2p^{-\theta -1}} & \text{if } p\equiv 3 \Mod 4), (\frac{m'}{p}) = 1 \vspace{.2cm}\\
\frac{3 - p^{-1} - 2p^{-\theta}}{1+p^{-1} - 2p^{-\theta -1}} & \text{if } p\equiv 1 \Mod 4), (\frac{m'}{p}) = -1 \vspace{.2cm}\\
1 & \text{if } p\equiv 3 \Mod 4), (\frac{m'}{p}) = -1.
\end{cases}
\end{equation*}
\end{lemma}

\begin{lemma}\label{Omega for p divide m in odd powers}
For $p^{2\theta-1} \mid\mid m$ with $\theta \geq 1$, we have
\begin{equation*}
\bm{\omega}(m, p) = \begin{cases}
\frac{3 - p^{-1} - p^{1-\theta} - p^{-\theta}}{1+p^{-1} - p^{-\theta} - p^{-\theta-1}} & \text{if } p\equiv 1 \Mod 4) \vspace{.2cm}\\
\frac{1 + p^{-1} - p^{1-\theta} - p^{-\theta}}{1+p^{-1} - p^{-\theta} - p^{-\theta-1}} & \text{if } p\equiv 1 \Mod 4).
\end{cases}
\end{equation*}
\end{lemma}

\subsection{Cusp form contribution}
We denote the cuspidal part associated to $Q_{d}$ by $f_{d}$. Let the $m$-th Fourier coefficient of $f_{d}$ be $\bm{R}(m, d)$, which can be expressed as
\begin{equation}\label{Rmd definition}
\bm{R}(m, d) = r_{Q_{d}}(m) - a_{E_{d}}(m).
\end{equation} 
The following lemma provides an upper bound of $\bm{R}(m, d)$.
\begin{lemma}\label{Error term}
For any $\epsilon >0$
$$\bm{R}(m, d) \ll d^{45/14}m^{13/28+\epsilon}$$
uniformly for $4d^2 \leq m^{1/2}$.
\end{lemma}
\begin{proof}
It follows from \cite[102:10]{Omeara} that the quadratic form $Q_{d}(\bm{\mathrm{x}}) = x_1^2 + x_2^2 + d^2 x_3^2$ with $\mu^2(d) =1$ and $2 \nmid d$, have only one spinor genus per genus. On the other hand, the last part of \cite[Theorem 1]{Blomer} yields that for any $\epsilon >0$,
$$r(\spn Q_{d}, m) - r_{Q_{d}}(m) \ll (4d^2)^{45/28} m^{13/28 + \epsilon}$$
uniformly for $4d^2 \leq m^{1/2}$. Hence these two facts together imply 
\begin{equation}\label{Upper bound}
r(\gen Q_{d}, m) - r_{Q_{d}}(m) \ll d^{45/14} m^{13/28 + \epsilon}.
\end{equation}
We can therefore conclude our lemma by inserting \eqref{Upper bound} into \eqref{eqn:SiegelWeil}.
\end{proof}

\section{Application of linear sieve}\label{S4}
Let $\mathscr{P}$ denotes the set of all odd primes. In this section, we seek estimates for the sifting function $S(\mathcal{A}, \mathscr{P}, z_0)$, which represents the number of elements in $\mathcal{A}$ that have no prime factors $p < z_0$ in $\mathscr{P}$. More formally, letting 
$$P(z_0) = \prod_{\substack{p<z_0\\ p \in \mathscr{P}}}p,$$
we want to estimate the following cardinality
\begin{equation*}
S(\mathcal{A}, \mathscr{P}, z_0) := \left|\{x_3 \in \mathcal{A} : (x_3, P(z_0)) = 1  \}\right|.
\end{equation*}
The following proposition provides an asymptotic formula for the cardinality of the set $\mathcal{A}_{d}$.
\begin{proposition}\label{Proposition 1}
Let $r_3(m)$ be the number of representation of $m$ with sum of three squares. Then for $2 \nmid d$, we have 
\begin{equation}\label{Asymptotic formula for Ad}
|\mathcal{A}_{d}| = \frac{\bm{\omega}(m, d)}{d}r_3(m) + \bm{R}(m, d),
\end{equation}
where $\bm{\omega}(m, d)$ and $\bm{R}(m, d)$ are defined in \eqref{Omega definition} and \eqref{Rmd definition} respectively.
For $0< \theta < 1/118$, the error term  $\bm{R}(m, d)$ satisfies
$$\sum_{d \leq m^\theta} \widetilde{\mu} (d)^2 |\bm{R}(m, d)| \ll m^{1/2 - \epsilon_2}$$
where $\widetilde{\mu} (d)$ is an arithmetic function defined as
\begin{equation*}
\widetilde{\mu} (d) = \begin{cases}
\mu (d) & \text{for } 2 \nmid  d, \\
0 & \text{otherwise,} 
\end{cases}
\end{equation*}
and $\epsilon_2>0$ is sufficiently small in terms of $\theta$.
\end{proposition}

\begin{proof}
The asymptotic formula \eqref{Asymptotic formula for Ad} of $|\mathcal{A}_{d}|$ is an easy consequence of \eqref{Main term} and \eqref{Rmd definition}.
It follows from Lemma \ref{Error term} that
$$\sum_{d \leq D} \widetilde{\mu} (d)^2 |\bm{R}(m, d)| \ll D^{59/14} m^{13/28 + \epsilon}.$$
Now the conditions $0< \theta < 1/118$ and $D = m^{\theta}$ will immediately imply
$$\sum_{d \leq D} \widetilde{\mu} (d)^2 |\bm{R}(m, d)| \ll m^{1/2 - \epsilon}.$$
This completes the proof the proposition.
\end{proof}
We define $$V(z_0) = \prod_{p\mid P(z_0)}\left(1 - \frac{\bm{\omega}(m, p)}{p}\right).$$
Let $F$ and $f$ denotes the classical functions of linear sieve which are the continuous solutions of the following system of differential-difference equations
\begin{align*}
sF(s) &= 2e^{\gamma} && \text{if } \quad 0<s \leq 3,&&&\hfill\\
sf(s) &= 0 && \text{if } \quad 0<s \leq 2,&&&\hfill\\
\left(sF(s)\right)' &= f(s-1) && \text{if } \quad s>3,&&&\hfill\\
\left(sf(s)\right)' &= F(s-1) && \text{if } \quad s>2,&&&\hfill
\end{align*}
where $\gamma$ is the Euler constant. The following proposition provides the upper and lower bound of the sifting function $S(\mathcal{A}, \mathscr{P}, z_0)$.
\begin{proposition}\label{Proposition for sifting function}
For $z_0 \geq 3$ and $D^2 \geq z_0$, we have
\begin{equation*}
S(\mathcal{A}, \mathscr{P}, z_0) \geq r_3(m)V(z_0)\left(f(s) + \mathcal{O}\left(e^{\sqrt{L-s}} (\log D)^{-1/3}\right)\right) - \sum_{\substack{d \leq D \vspace{.05cm}\\ d \mid P(z_0)}} \widetilde{\mu} (d)^2\left|\bm{R}(m, d)\right|
\end{equation*}
and
\begin{equation*}
S(\mathcal{A}, \mathscr{P}, z_0) \leq r_3(m)V(z_0)\left(F(s) + \mathcal{O}\left(e^{\sqrt{L-s}} (\log D)^{-1/3}\right)\right) + \sum_{\substack{d \leq D \vspace{.05cm}\\ d \mid P(z_0)}} \widetilde{\mu} (d)^2\left|\bm{R}(m, d)\right|
\end{equation*}
where $L$ is an absolute constant and $s = \frac{\log D}{\log z_0}$.
\end{proposition}
We prove the proposition by using Rosser's weights (cf. \cite{Iwaniec1}, \cite{Iwaniec2}). Fixing a positive integer $D$, we define two sequences $\{\lambda_d^{\pm}\}$ in a following way.
\begin{itemize}
\item[(i)]
$\lambda_1^{\pm} = 1$.
\item[(ii)]
$\lambda_d^{\pm} = 0$ if $d$ is not square-free.
\item[(iii)]
For $d= p_1 p_2 \cdots p_r$ with $p_1>p_2> \cdots >p_r$ and $2, 5 \nmid d$
\begin{equation*}
\lambda_d^{+} = \begin{cases}
(-1)^r & \text{ if } p_1 \cdots p_{2l} \ p_{2l+1}^3 <D \text{ whenever } 0 \leq l \leq \frac{r-1}{2}\\
0 & \text{ Otherwise }
\end{cases} 
\end{equation*}
and
\begin{equation*}
\lambda_d^{-} = \begin{cases}
(-1)^r & \text{ if } p_1 \cdots p_{2l-1} \ p_{2l}^3 <D \text{ whenever } 0 \leq l \leq \frac{r}{2}\\
0 & \text{ Otherwise }.
\end{cases} 
\end{equation*}
\end{itemize}
It can be deduced from Lemma \ref{Omega for p does not divide m}, Lemma \ref{Omega for p divide m in even powers} and Lemma \ref{Omega for p divide m in odd powers} that $\bm{\omega}(m, p)$ satisfies the following two inequalities, which are
\begin{equation}\label{Linear sieve condition 1}
0 \leq \bm{\omega}(m, p) < p
\end{equation}
for all primes $p$ and there exist an absolute constant $L$ independent of $m$ such that
\begin{equation}\label{Linear sieve condition 2}
\prod_{w \leq p < z_0} \left(1- \frac{\bm{\omega}(m, p)}{p}\right)^{-1} < \left(\frac{\log z_0}{\log w}\right)\left(1+ \frac{L}{\log w}\right)
\end{equation}
for every $2\leq w <z_0$.
 
Hence it follows from \cite[Lemma 3]{Iwaniec2} that the inequalities \eqref{Linear sieve condition 1} and \eqref{Linear sieve condition 2} yields the following Lemma which is crucial to obtain the upper and lower bound of the main term of the sifted function. 
\begin{lemma}\label{Lemma of Iwaniec}
We have 
\begin{equation*}
V(z_0) \geq \sum_{d \mid P(z_0)} \lambda_d^{-}\frac{\bm{\omega}(m, d)}{d} \geq V(z_0)\left(f(s) + \mathcal{O}\left(e^{\sqrt{L-s}} (\log D)^{-1/3}\right)\right)
\end{equation*}
whenever $z_0 \leq D^{1/2}$ and
\begin{equation*}
V(z_0) \leq \sum_{d \mid P(z_0)} \lambda_d^{+}\frac{\bm{\omega}(m, d)}{d} \leq V(z_0)\left(F(s) + \mathcal{O}\left(e^{\sqrt{L-s}} (\log D)^{-1/3}\right)\right)
\end{equation*}
whenever $z_0 \leq D$. Here $L$ is an absolute constant arising from \eqref{Linear sieve condition 2} and $s = \frac{\log D}{\log z_0}$.
\end{lemma}

\subsection{Proof of the Proposition \ref{Proposition for sifting function}}
The basic inclusion-exclusion principle yields
\begin{equation}\label{Inclusion Exclusion}
S(\mathcal{A}, \mathscr{P}, z_0) = \sum_{d \mid P(z_0)} \widetilde{\mu} (d)|\mathcal{A}_{d}|
\end{equation}
Inserting the cardinality of the set $\mathcal{A}_{d}$ from Proposition \ref{Proposition 1} into \eqref{Inclusion Exclusion}, we obtain
\begin{equation*}
S(\mathcal{A}, \mathscr{P}, z_0) = \sum_{d \mid P(z_0)} \widetilde{\mu} (d)\frac{\bm{\omega}(m, d)}{d}r_3(m) + \sum_{d \mid P(z_0)} \widetilde{\mu} (d) \bm{R}(m, d).
\end{equation*}
Thus, Lemma \ref{Lemma of Iwaniec} yields for $z_0 \leq D^{1/2}$, the main term of $S(\mathcal{A}, \mathscr{P}, z_0)$ can be bounded as 
\begin{align}\label{Proposition main term}
r_3(m) V(z_0)\left(f(s) + O\left(e^{\sqrt{L-s}} (\log D)^{-1/3}\right)\right) \leq \sum_{d \mid P(z)}& \widetilde{\mu} (d)\frac{\bm{\omega}(m, d)}{d}r_3(m)\nonumber \\
& \leq r_3(m) V(z_0)\left(F(s) + O\left(e^{\sqrt{L-s}} (\log D)^{-1/3}\right)\right).
\end{align}
The error term of $S(\mathcal{A}, \mathscr{P}, z_0)$ can be estimated using Rosser's weights as 
\begin{equation*}
\sum_{d \mid P(z_0)} \widetilde{\mu} (d) \bm{R}(m, d) \leq \sum_{d \mid P(z_0)} \lambda_d^{+} \bm{R}(m, d)
\end{equation*}
Now observe that  $\lambda_d^{+} = 0$ for $d> D$. Hence we have
\begin{equation*}
\sum_{d \mid P(z_0)} \widetilde{\mu} (d) \bm{R}(m, d) \leq \sum_{\substack{d \leq D \vspace{.05cm}\\ d \mid P(z_0)}} \lambda_d^{+} \bm{R}(m, d)
\end{equation*}
We can therefore provide the bound for the absolute value of the error term as
\begin{equation}\label{Proposition error term}
\left|\sum_{d \mid P(z_0)} \widetilde{\mu} (d) \bm{R}(m, d)\right| \leq \sum_{\substack{d \leq D \vspace{.05cm}\\ d \mid P(z_0)}}\left| \lambda_d^{+}\right| \left|\bm{R}(m, d)\right| \leq \sum_{\substack{d \leq D \vspace{.05cm}\\ d \mid P(z_0)}} \widetilde{\mu} (d)^2\left|\bm{R}(m, d)\right|.
\end{equation}
Thus \eqref{Proposition main term} and \eqref{Proposition error term} together concludes the proof of the proposition.

\subsection{Proof of Theorem \ref{Thm 1}}
Letting $\theta < 1/118$ and fixing $D = m^{\theta}$, the Proposition \ref{Proposition 1} yields that the error term in Proposition \ref{Proposition for sifting function} can be bounded as 
\begin{equation}\label{Bound for error term}
\sum_{d \leq D} \widetilde{\mu} (d)^2 |\bm{R}(m, d)| \ll m^{1/2 - \epsilon_2}.
\end{equation}
On the other hand, for $z_0 = m^\gamma$ with $\gamma < 1/236$, Lemma \ref{Lemma of Iwaniec} implies that 
\begin{equation}\label{Bound for main term}
V(z_0) \gg \frac{1}{\log m}
\end{equation}
where $s = \frac{\log D}{\log z_0} > 2$. Applying the bounds \eqref{Siegel's bound}, \eqref{Bound for error term} and \eqref{Bound for main term} together in Proposition \ref{Proposition for sifting function}, we see that in order to prove Theorem \ref{Thm 1} it remains to show that  $f(s) >0$, where $f$ is a classical function of linear sieve. A numerical calculation shows that $f(s) >0$ for $s >2$. This completes the proof of Theorem \ref{Thm 1}.

\section{Representation of every integer by certain quaternary quadratic forms}\label{S5}
The set to be sieved is the following 
 $$\mathcal{A} := \{(z_1, z_2) \in \N^2 : x^2 + y^2 + (2^a 3^bz_1)^2 + (2^c 5^d z_2)^2= m \},$$
 where $x, y, a, b, c, d$ are any integers with $a$, $b$, $c$ and $d$ non-negative. Setting $x_1 =x$, $x_2 = y$ $x_3 = 2^a 3^bz_1$ and $x_4 = 2^c 5^dz_2$, we can write
$$\mathcal{A} = \{(x_3, x_4) \in \N^2 : x_1^2 + x_2^2 + x_3^2 + x_4^2= m \}.$$ 
For $\widetilde{x} = (x_3, x_4)$ and $\bm{\mathrm{d}} = (d_3, d_4)$, let $\widetilde{x} \equiv 0 \Mod \bm{\mathrm{d}})$ denotes the simultaneous conditions $x_3 \equiv 0 \Mod d_3)$ and $x_4 \equiv 0 \Mod d_4)$. In order to apply sieve theory, we need an asymptotic formula for the cardinality of the following set
$$\mathcal{A}_{\bm{\mathrm{d}}} := \{\widetilde{x} \in \mathcal{A} : \widetilde{x} \equiv 0 \Mod \bm{\mathrm{d}}) \} = \{(x_3, x_4) \in \N^2 : x_1^2 + x_2^2 + d_3^2 x_3^2 + d_4^2 x_4^2 = m \},$$
where $2, 3 \nmid d_3$ and $2, 5 \nmid d_4$. 
It needs some preparation to express the above cardinality in terms of main term and error term. We first consider the quadratic form
$$Q_{\bm{\mathrm{d}}}(\bm{\mathrm{x}}) = x_1^2 + x_2^2 + d_3^2 x_3^2 + d_4^2 x_4^2$$ 
where $\bm{\mathrm{x}} = (x_1, x_2, x_3, x_4)$ and $\bm{\mathrm{d}}=(d_3, d_4)$ satisfies $2,3 \nmid d_3$ and $2,5 \nmid d_4$. For simplicity, we abbreviate $Q_{(1,1)} = Q$. Let, the theta function associated to $Q_{\bm{\mathrm{d}}}$ be $\Theta_{Q_{\bm{\mathrm{d}}}}$ with the Fourier expansion 
$$\Theta_{Q_{\bm{\mathrm{d}}}}(\tau) = \sum_{m\geq 0} r_{Q_{\bm{\mathrm{d}}}}(m) q^n$$
for $\tau \in \H$. Here the Fourier coefficient $r_{Q_{\bm{\mathrm{d}}}}(m)$ denotes the number of representation of $m$ by  $Q_{\bm{\mathrm{d}}}$. As mentioned earlier, $\Theta_{Q_{\bm{\mathrm{d}}}}$ can be decomposed into two parts which are the Eisenstein series part and the cuspidal part respectively.  

\subsection{Eisenstein series contribution}
We denote the Eisenstein series part associated to $Q_{\bm{\mathrm{d}}}$ by $E_{\bm{\mathrm{d}}}$. It follows from \eqref{Local density} that the $m$-th Fourier coefficient of $E_{d}$ can be expressed as 
\begin{equation}\label{Eisenstein product : 2}
a_{E_{\bm{\mathrm{d}}}}(m) = \prod_{p} \beta_{Q_{\bm{\mathrm{d}},} p}(m)
\end{equation}  
where the product runs over all the primes including $\infty$. The definition \eqref{Def local density} of local representation density yields 
\begin{equation*}
\beta_{Q_{\bm{\mathrm{d}},} p}(m) = \lim_{r \to \infty}\frac{|R_{Q_{\bm{\mathrm{d}},} p^r}(m)|}{p^{3r}}
\end{equation*}
where $$R_{Q_{\bm{\mathrm{d}},} p^r}(m) := \{ \x \in (\Z/p^r \Z)^4 : Q_{\bm{\mathrm{d}}}(\x) \equiv m \Mod p^r) \}.$$ The following lemma is crucial to determine the local density of $Q_{\bm{\mathrm{d}}}$ at $p =2$. 
\begin{lemma}\label{Lemma for p not divide d : 2}
For $p \nmid d_3d_4$, the local density of $Q_{d}$ satisfies
$$\beta_{Q_{\bm{\mathrm{d}},} p} = \beta_{Q, p}.$$
In particular, we have $$\beta_{Q_{\bm{\mathrm{d}},} 2} = \beta_{Q, 2}.$$
\end{lemma}
\begin{proof}
The proof of the lemma follows similarly as of the proof of Lemma \ref{Lemma for p not divide d}.
\end{proof}
We set $d_j = p^{\alpha_j}d_j'$ with $p\nmid d_j'$ and $m=p^Rm'$ with $p\nmid m'$. In the next lemma, we compute the local representation density $\beta_{Q_{\bm{\mathrm{d}},} p}$ at each odd prime $p$.
\begin{lemma}\label{Lemma for local computation:2}
Let $\alpha_3\leq \alpha_4$. Then for any odd prime $p$, we have
\begin{equation*}
\beta_{Q_{\bm{\mathrm{d}},} p} (m)= 
\begin{cases}
1 + (1 - p^{-1})\left( \left[ \frac{R}{2} \right] + \varepsilon_p^2 \left[\frac{R+1}{2}\right]\right)- \delta_{2\nmid R} \, p^{-1} - \delta_{2\mid R} \, \varepsilon_p^2 \, p^{-1}
& \text{if  } R < 2 \alpha_3,\vspace{.5cm}\\
\begin{aligned}
1 + (1 + \varepsilon_p^2)&(1-p^{-1})\alpha_3 + p^{-1} - p^{\alpha_3 -1- [\frac{R}{2}]}\\
& - \delta_{2\nmid R} \, p^{\alpha_3-1 - \frac{R+1}{2}}
+ \delta_{2\mid R} \, \varepsilon_p^2 \, p^{\alpha_3-1 - \frac{R}{2}} \left(\frac{m'}{p} \right)
\end{aligned}
&  \text{if  }  2 \alpha_3 \leq R < 2\alpha_4, \vspace{.5cm}\\
1 + (1 + \varepsilon_p^2)(1-p^{-1})\alpha_3 + p^{-1} - p^{\alpha_3 +\alpha_4 -1-R}(1+p^{-1})
& \text{if  } R \geq 2 \alpha_4.
\end{cases}
\end{equation*}
\end{lemma}
\begin{proof}
We first evaluate the absolute value of $R_{Q_{\bm{\mathrm{d}},} p^r}(m)$. It follows from \eqref{Orthogonality of roots of unity} that
\begin{align*}
|R_{Q_{\bm{\mathrm{d}},} p^r}(m)|&=\sum_{\substack { \x \in(\Z/p^r\Z)^4\\ Q_{\bm{\mathrm{d}}}(\bm{\mathrm{x}}) \, \equiv \, m \Mod p^r)}}1\\
& = \sum_{\x \in(\Z/p^r\Z)^4}\frac{1}{p^r} \sum_{n \Mod p^r)} e^{\frac{2\pi i n}{p^r}\left(Q_{\bm{\mathrm{d}}}(\bm{\mathrm{x}})-m\right)}\\
&= \frac{1}{p^r} \sum_{n\Mod p^r)} e^{-\frac{2\pi i nm}{p^r}} \sum_{\x\in(\Z/p^r\Z)^4}e^{\frac{2\pi i n}{p^r}Q_{\bm{\mathrm{d}}}(\bm{\mathrm{x}})}\\
&= \frac{1}{p^r} \sum_{n\Mod p^r)} e^{-\frac{2\pi i nm}{p^r}}\(\prod_{j=1}^2 \sum_{x_j\in \Z/p^r\Z}e^{\frac{2\pi i nx_j^2}{p^r}} \right) \left(\prod_{j=3}^4\sum_{x_j \in \Z/p^r\Z}e^{\frac{2\pi i nd_j^2x_j^2}{p^r}}\right)\\
&=\frac{1}{p^r} \sum_{n\Mod p^r)} e^{-\frac{2\pi i nm}{p^r}} G_2(n, 0, p^r)^2 \prod_{j=3}^4 G_2(nd_j^2, 0, p^r)
\end{align*}
We next split the sum over $n$ by writing $n=p^kn'$ with $p\nmid n'$ and make the change of variables $k\mapsto r-k$ to arrive at
\begin{align*}
|R_{Q_{\bm{\mathrm{d}},} p^r}(m)| &= \frac{1}{p^r} \sum_{k=0}^r p^{4k}\sum_{n'\in (\Z/p^{r-k}\Z)^{\times}} e^{-\frac{2\pi i n'm}{p^{r-k}}} G_2(n', 0, p^{r-k})^2 \prod_{j=3}^4G_2(n'd_j^2, 0, p^{r-k})\\
&= p^{3r} \sum_{k=0}^r p^{-4k}\sum_{n'\in (\Z/p^{k}\Z)^{\times}} e^{-\frac{2\pi i n'm}{p^{k}}}G_2(n', 0, p^{k})^2 \prod_{j=3}^4G_2(n'd_j^2, 0, p^{k}), 
\end{align*}
where in the last step we applied \eqref{eqn:G2gcd}. We can derive similarly as in \eqref{G_2 evaluation} that
\begin{align*}
G_2(n'd_j^2, 0, p^{k}) = \begin{cases}
p^k & \text{for } k\leq 2\alpha_j, \\
p^{2\alpha_j}G_2(n'd_j^2, 0, p^{k-2\alpha_j}) & \text{for } k>2\alpha_j.
\end{cases}
\end{align*}
Therefore for the assumption $\alpha_3 \leq \alpha_4$, we can write
\begin{align*}
\frac{1}{p^{3r}}|R_{Q_{\bm{\mathrm{d}},} p^r}(m)| &= \sum_{k=0}^{2 \alpha_3} p^{-2k}\sum_{n'\in (\Z/p^{k}\Z)^{\times}} e^{-\frac{2\pi i n'm}{p^{k}}}G_2(n', 0, p^{k})^2 \nonumber\\
&+ \sum_{k=2\alpha_3+1}^{2\alpha_4} p^{-3k + 2\alpha_3} \sum_{n'\in (\Z/p^{k}\Z)^{\times}} e^{-\frac{2\pi i n'm}{p^{k}}}G_2(n', 0, p^{k})^2G_2(n'd_3^2, 0, p^{k-2\alpha_3})\nonumber\\
&+ \sum_{k=2\alpha_4+1}^{r} p^{-4k + 2\alpha_3 + 2\alpha_4} \sum_{n'\in (\Z/p^{k}\Z)^{\times}} e^{-\frac{2\pi i n'm}{p^{k}}}G_2(n', 0, p^{k})^2\prod_{j=3}^4 G_2(n'd_j^2, 0, p^{k-2\alpha_j}).
\end{align*}
Thus the evaluation of quadratic Gauss sums in \eqref{eqn:G2(1)} and \eqref{eqn:G2(2)} yields
\begin{align*}
\frac{1}{p^{3r}}|R_{Q_{\bm{\mathrm{d}},} p^r}(m)| &= \sum_{k=0}^{2\alpha_3}(\delta_{2\mid k} + \delta_{2\nmid k}\varepsilon_p^2)\, p^{-k} \tau(\chi_{1, p^k}, \psi_{-m, p^k}) 
+ \sum_{\substack{k=2\alpha_3+1 \\ k \equiv 0 \Mod 2)}}^{2\alpha_4} p^{\frac{-3k}{2}+\alpha_3} \tau(\chi_{1, p^k}, \psi_{-m, p^k}) \nonumber\\
&+ \varepsilon_p^3 \sum_{\substack{k=2\alpha_3+1 \\ k \equiv 1 \Mod 2)}}^{2\alpha_4} p^{\frac{-3k}{2}+\alpha_3} \tau(\chi_{p, p^k}, \psi_{-m, p^k})+ \sum_{k=2\alpha_4+1}^{r}p^{-2k+\alpha_3+\alpha_4} \tau(\chi_{1, p^k}, \psi_{-m, p^k}).
\end{align*}
Inserting the value of $\tau(\chi_{1, p^k}, \psi_{-m, p^k})$ and $\tau(\chi_{p, p^k}, \psi_{-m, p^k})$ from \eqref{eqn:tau1} and \eqref{eqn:taup} in the above equation, it can be derived for $R < 2\alpha_3$ that
\begin{align*}
\lim_{r \to \infty} \frac{1}{p^{3r}}|R_{Q_{\bm{\mathrm{d}},} p^r}(m)| &= 1 + \sum_{k=1}^R (\delta_{2\mid k} + \delta_{2\nmid k} \, \varepsilon_p^2) p^{-k} (p^k - p^{k-1}) - \delta_{2\nmid R} \, p^{-1} - \delta_{2\mid R} \, \varepsilon_p^2 \, p^{-1}\nonumber\\
&= 1 + (1 - p^{-1})\left( \left[ \frac{R}{2} \right] + \varepsilon_p^2 \left[\frac{R+1}{2}\right]\right) - \delta_{2\nmid R} \, p^{-1} - \delta_{2\mid R} \, \varepsilon_p^2 \, p^{-1}.
\end{align*}
For $2\alpha_3 \leq R < 2\alpha_4$, we have
\begin{align*}
\lim_{r \to \infty} \frac{1}{p^{3r}}|R_{Q_{\bm{\mathrm{d}},} p^r}(m)| = 1 + (1 + \varepsilon_p^2)(1-p^{-1})\alpha_3 &+ \sum_{\substack{k = 2\alpha_3 + 1 \\ k \equiv 0 \Mod 2)}}^R p^{\frac{-3k}{2}+\alpha_3}(p^k - p^{k-1}) \\
&- \delta_{2\nmid R} \, p^{\alpha_3 - \frac{R+3}{2}}
+ \delta_{2\mid R} \, \varepsilon_p^2 \, p^{\alpha_3 - \frac{R}{2} -1} \chi_p(m')
\\
=  1 + (1 + \varepsilon_p^2)(1-p^{-1})\alpha_3 &+ p^{-1} - p^{\alpha_3 -1- [\frac{R}{2}]}\\ 
&- \delta_{2\nmid R} \, p^{\alpha_3-1 - \frac{R+1}{2}}
+ \delta_{2\mid R} \, \varepsilon_p^2 \, p^{\alpha_3-1 - \frac{R}{2}} \chi_p(m').
\end{align*}
Finally we consider the case $R \geq 2\alpha_4$ and obtain
\begin{align*}
\lim_{r \to \infty} \frac{1}{p^{2r}}|R_{Q_{\bm{\mathrm{d}},} p^r}(m)| &= 1 + (1 + \varepsilon_p^2)(1-p^{-1})\alpha_3 + p^{-1} - p^{\alpha_3 -\alpha_4 -1} \\
&+ \sum_{k=2\alpha_4+1}^R p^{-2k+\alpha_3+\alpha_4}(p^k-p^{k-1}) - p^{-2(R+1)+\alpha_3+\alpha_4+R}\\
&= 1 + (1 + \varepsilon_p^2)(1-p^{-1})\alpha_3 + p^{-1} - p^{\alpha_3 +\alpha_4 -1-R} - p^{\alpha_3 +\alpha_4 -2-R}
\end{align*}
This concludes the proof of the lemma.
\end{proof}
The next lemma relates the local densities of $Q_{\bm{\mathrm{d}}}$ and $Q$ at $p = \infty$.
\begin{lemma}\label{Lemma Density at infinity:2}
We have 
\begin{equation*}
\beta_{Q_{\bm{\mathrm{d}}}, \infty} = \frac{1}{d_3d_4}\beta_{Q, \infty}.
\end{equation*}
\end{lemma}
\begin{proof}
The proof of the lemma follows similarly as of the proof of Lemma \ref{Lemma Density at infinity}.
\subsection{Main term computation}\label{Definition of omega:2}
We next define the multiplicative function $\bm{\omega}(m, \bm{\mathrm{d}})$ for each square-free $d_3$ and $d_4$ with $2, 3 \nmid d_3$ and $2, 5 \nmid d_4$ by the following :
\begin{equation}\label{Omega definition:2}
\bm{\omega}(m, \bm{\mathrm{d}}) := \prod_{p^{\nu} || d_3 d_4} \bm{\omega}_\nu(m, p)
\end{equation}
where $$\bm{\omega}_\nu(m, p) = \frac{\beta_{Q_{\bm{\mathrm{d}},} p}(m)}{\beta_{Q, p}(m)}.$$ 
Here $\nu$ can take the values $1$ or $2$. Lemma \ref{Lemma for p not divide d : 2} together with Lemma \ref{Lemma Density at infinity:2} implies that \eqref{Eisenstein product : 2} can be rephrased as
\begin{equation*}
a_{E_{\bm{\mathrm{d}}}}(m) = \frac{a_{E_{\bm{\mathrm{d}}}}(m)}{a_{E}(m)} a_{E}(m)= \frac{ \bm{\omega}(m, \bm{\mathrm{d}})}{d_3 d_4}r_4(m).
\end{equation*}
We can now proceed to find explicit formulas for $\bm{\omega}_\nu(m, p)$. Invoking Lemma \ref{Lemma for local computation:2} inside the definition of $\bm{\omega}_\nu(m, p)$ the following evaluations can be obtained.
\end{proof}

\begin{lemma}\label{Omega for p does not divide m:2}
For $p\nmid m$, we have 
\begin{equation*}
\bm{\omega}_{\nu}(m, p) = \begin{cases}
\frac{1+\varepsilon_p^2\chi_p(m')p^{-1}}{1-p^{-2}} & \text{for } \nu = 1 \vspace{.2cm}\\
\frac{1-\varepsilon_p^2 p^{-1}}{1-p^{-2}} & \text{for } \nu = 2. 
\end{cases}
\end{equation*}
\end{lemma}

\begin{lemma}\label{Omega for p divide m:2}
For $p^R || m$ with $R\geq 1$, we have
\begin{equation*}
\bm{\omega}_{\nu}(m, p) = \begin{cases}
\frac{1-p^{-R}}{1-p^{-R-1}} & \text{for } \nu = 1 \vspace{.2cm}\\
\frac{(1-p^{-1})(1+\varepsilon_p^2)+(1+p^{-1})(1-p^{1-R})} {(1+p^{-1})(1-p^{-R-1})} & \text{for } \nu = 2. 
\end{cases}
\end{equation*}
\end{lemma}

We next define another function
\[
\Omega(m, p) := 2 \bm{\omega}_1(m, p) - \frac{\bm{\omega}_2(m, p)}{p}
\]
and 
\[
\Omega(m, d) := \prod_{p\mid d} \Omega(m, p). 
\]
In the sieve theory calculation, the main term involves the following term : 
\[
W(z_0) := \prod_{p<z_0} \left(1-\frac{\Omega(m, p)}{p} \right).
\]
Therefore, we finally need the lower bound of $W(z_0)$ i.e, the upper bound of $\Omega(m, p)$. 
\begin{lemma}\label{Big omega bound}
We have
\[
\Omega(m, p) \leq \begin{cases}
2 + \frac{p+3}{p^2 -1} & \text{ for } \, p \nmid m\\
2 & \text{ for } \, p \mid m.
\end{cases}
\]
In particular, for $p \geq 5$ we have $\Omega(m, p) \leq \frac{5}{2}$. 
\end{lemma}
\begin{proof}
We first apply Lemma  \ref{Omega for p does not divide m:2} and Lemma \ref{Omega for p divide m:2} to obtain trivial bounds of $\bm{\omega}_\nu(m, p)$ and then combine and simplify the bounds to conclude the lemma.
\end{proof}

\subsection{Cusp form contribution}
We denote the cuspidal part associated to $Q_{\bm{\mathrm{d}}}$ by $f_{\bm{\mathrm{d}}}$. Let the $m$-th Fourier coefficient of $f_{\bm{\mathrm{d}}}$ be $\bm{R}(m, \bm{\mathrm{d}})$, which can be expressed as
\begin{equation}\label{Rmd definition:2}
\bm{R}(m, \bm{\mathrm{d}}) = r_{Q_{\bm{\mathrm{d}}}}(m) - a_{E_{\bm{\mathrm{d}}}}(m).
\end{equation}
The following lemma provides an upper bound of $\bm{R}(m, \bm{\mathrm{d}})$.
\begin{lemma}\label{Upper bound on Rmd:2}
We have
\begin{equation*}
|\bm{R}(m, \bm{\mathrm{d}})| \leq  7.07 \times 10^{23} (d_3 d_4)^{6.01} m^{\frac{3}{5}}.
\end{equation*}
\end{lemma}
\begin{proof}
We write the quadratic form $Q_{\bm{\mathrm{d}}}(\bm{\mathrm{x}}) = \frac{1}{2}\bm{\mathrm{x}}^{T} A \bm{\mathrm{x}}$ where $\bm{\mathrm{x}} = (x_1, x_2, x_3, x_4)$ and $A$ is the diagonal matrix such that $A = [2, 2, 2d_3^2, 2d_4^2]$. 
For $N_{\bm{\mathrm{d}}}$ and $\Delta_{\bm{\mathrm{d}}}$ denoting the level and the discriminant respectively of the quadratic forms  $Q_{\bm{\mathrm{d}}}$, we have 
\begin{equation}\label{Disc and level}
N_{\bm{\mathrm{d}}} = 4 \lcm (d_3^2, d_4^2)  \hspace{.4cm} \text{and} \hspace{.4cm} \Delta_{\bm{\mathrm{d}}} = 2^4 d_3^2 d_4^2.
\end{equation}
It follows from \cite[Lemma 4.2]{BK} that
\begin{equation*}
|\bm{R}(m, \bm{\mathrm{d}})| \leq  1.797 \times 10^{21} m^{\frac{3}{5}} N_{\bm{\mathrm{d}}}^{\frac{3}{2}+ 2 \cdot 10^{-6}+ \frac{1}{200}}(27 \pi \Delta_{\bm{\mathrm{d}}} + 16 N_{\bm{\mathrm{d}}}^3)^{\frac{1}{2}}.
\end{equation*}
Thus by inserting the value of $N_{\bm{\mathrm{d}}}$ and $\Delta_{\bm{\mathrm{d}}}$ from \eqref{Disc and level} into the above equation, we can bound the absolute value of $\bm{R}(m, \bm{\mathrm{d}})$ as
\begin{align*}
|\bm{R}(m, \bm{\mathrm{d}})| \leq 1.448 \times 10^{22} m^{\frac{3}{5}} (\lcm (d_3^2, d_4^2))^{\frac{3}{2}+ 2 \cdot 10^{-6}+ \frac{1}{200}}(432 \pi  d_3^2 d_4^2 + 1024 (\lcm (d_3^2, d_4^2))^3)^{\frac{1}{2}}.
\end{align*}
The trivial bound $\lcm (d_3^2, d_4^2) \leq d_3^2d_4^2$ yields
\begin{align*}
|\bm{R}(m, \bm{\mathrm{d}})|&\leq 5.791 \times 10^{22} m^{\frac{3}{5}} (d_3d_4)^{4+4\cdot 10^{-6}+\frac{1}{100}}(27\pi + 64 d_3^4 d_4^4)^{\frac{1}{2}}\\
&\leq 5.791 \times 10^{22} (27\pi + 64)^{\frac{1}{2}} m^{\frac{3}{5}} (d_3d_4)^{6+4\cdot 10^{-6}+\frac{1}{100}}\\
&\leq 7.07 \times 10^{23} m^{\frac{3}{5}} (d_3d_4)^{6.01}.
\end{align*}
This completes the proof of the lemma.
\end{proof}

\section{Application of vector sieve}\label{S6}
Let $\mathscr{P}_7$ be the set of all primes starting from $7$.
For
$$P(z_0) := \prod_{\substack{p< z_0 \\ p \in \mathscr{P}_7}} p,$$
we let $P_3(z_0) := 5 P(z_0)$ and $P_4(z_0) := 3 P(z_0)$
and seek estimates for the sifting function
\begin{equation*}
S(\mathcal{A}, \mathscr{P}_7, z_0) := \left|\{(x_3, x_4) \in \mathcal{A} : (x_j, P_j(z_0)) = 1 \, \text{ for }  j= 3, 4\}\right|.
\end{equation*}
In the following proposition, we provide an asymptotic expansion for the cardinality of the set $\mathcal{A}_{\bm{\mathrm{d}}}$, which follows from \eqref{Omega definition:2} and \eqref{Rmd definition:2}.

\begin{proposition}\label{Proposition 1:2}
Let $r_4(m)$ be the number of representation of $m$ with sum of four squares. Then for $2,3 \nmid d_3$ and $2,5 \nmid d_4$, we have 
\begin{equation*}
|\mathcal{A}_{\bm{\mathrm{d}}}| = \frac{\bm{\omega}(m, \bm{\mathrm{d}})}{d_3d_4}r_4(m) + \bm{R}(m, \bm{\mathrm{d}}).
\end{equation*}
\end{proposition}
The following bound on $\bm{\omega}_1(m, p)$ is crucial to prove Theorem \ref{Thm 2}.
\begin{lemma}\label{Boundinlog}
For $3<w\leq z_0$, we have
\begin{equation*}
\prod_{w\leq p < z_0}\left(1- \frac{\bm{\omega}_1(m, p)}{p}\right)^{-1} \leq \left(\frac{\log z_0}{\log w} \right)\left(1+\frac{6}{\log w} \right)
\end{equation*}
\end{lemma}
\begin{proof}
The evaluation of $\bm{\omega}_1(m, p)$ in Lemma \ref{Omega for p does not divide m:2} and Lemma \ref{Omega for p divide m:2} implies that 
\begin{equation}\label{omega1 bound:2}
\bm{\omega}_1(m, p) \leq \frac{p}{p-1}.
\end{equation}
We can therefore bound
\begin{align}\label{Eulerprodomegabound:2}
\prod_{w\leq p < z_0}\left(1- \frac{\bm{\omega}_1(m, p)}{p}\right)^{-1} \leq \prod_{w\leq p < z_0}\left(1- \frac{1}{p-1}\right)^{-1} \leq \prod_{w\leq p < z_0} \frac{p}{p-1} \prod_{w\leq p < z_0} \left(1+ \frac{3}{p^2}\right)^{-1},
\end{align}
where in the last inequality we used the bound $\frac{1}{p(p-2)}\leq \frac{3}{p^2}$ for $p\geq 3$. We next bound both the products of the above equation separately. It follows from \cite[(3.30) and (3.26)]{Rosser} that
\begin{equation}\label{ApplyRosserSchoenfield:2}
\prod_{w\leq p < z_0} \frac{p}{p-1} = \left(\prod_{p < z_0} \frac{p}{p-1} \right) \left( \prod_{p < w} \frac{p-1}{p}\right)\leq \left(\frac{\log z_0}{\log w} \right) \left(1+\frac{1}{\log^2 z_0}\right) \left(1+\frac{1}{2\log^2 w}\right)
\end{equation}
Let $\omega(n)$ denote the number of distinct prime divisors of $n$. We have
\[
\prod_{w\leq p < z_0} \left(1+ \frac{3}{p^2}\right)^{-1} \leq \prod_{p \geq w} \left(1+ \frac{3}{p^2}\right)^{-1} = 1+ \sum_{\substack{n>1\\ p\mid n \implies p\geq w}} \frac{\mu^2(n)3^{\omega(n)}}{n^2}
\]
Applying the bound $3^{\omega(n)} \leq 1.614 n^{1/2}$ (cf. \cite[Lemma 2.5]{BK}) and bounding the sum against the Riemann integral, the above bound can be reduced as
\begin{equation}\label{Prod3omegabound:2}
\prod_{w\leq p < z_0} \left(1+ \frac{3}{p^2}\right)^{-1} \leq 1+ 1.614 \sum_{n\geq w} \frac{1}{n^{3/2}} \leq 1+ \frac{3.228}{\sqrt{w}}
\end{equation}
We next insert the bounds \eqref{ApplyRosserSchoenfield:2} and \eqref{Prod3omegabound:2} of both the products into \eqref{Eulerprodomegabound:2} to obtain
\begin{equation}\label{Astepbeforelemma}
\prod_{w\leq p < z_0}\left(1- \frac{\bm{\omega}_1(m, p)}{p}\right)^{-1} \leq \left(\frac{\log z_0}{\log w} \right) \left(1+\frac{1}{\log^2 z_0}\right) \left(1+\frac{1}{2\log^2 w}\right) \left(1+ \frac{3.228}{\sqrt{w}}\right)  
\end{equation}
Finally, we utilize the bound 
\[
\left(1+\frac{1}{\log^2 z_0}\right) \left(1+\frac{1}{2\log^2 w}\right) \left(1+ \frac{3.228}{\sqrt{w}}\right) \leq \left(1 +\frac{6}{\log w}\right)
\]
in \eqref{Astepbeforelemma}, to conclude our lemma.
\end{proof}
For $\beta, D>0$, we define two sequences $\{\lambda_{d}^{\pm}\}$ in a following way.
\begin{itemize}
\item[(i)]
$\lambda_1^{\pm} = 1$.
\item[(ii)]
$\lambda_{d}^{\pm} = 0$ if $d$ is not square-free.
\item[(iii)]
For $d= p_1 p_2 \cdots p_r$ with $p_1>p_2> \cdots >p_r$ 
\begin{equation*}
\lambda_{d}^{+} = \begin{cases}
(-1)^r & \text{ if } p_1 \cdots p_{2l} \ p_{2l+1}^{\beta +1} <D \text{ whenever } 0 \leq l \leq \frac{r-1}{2}\\
0 & \text{ Otherwise }
\end{cases} 
\end{equation*}
and
\begin{equation*}
\lambda_{d}^{-} = \begin{cases}
(-1)^r & \text{ if } p_1 \cdots p_{2l-1} \ p_{2l}^{\beta +1} <D \text{ whenever } 0 \leq l \leq \frac{r}{2}\\
0 & \text{ Otherwise }.
\end{cases} 
\end{equation*}
\end{itemize}
We define 
\begin{equation*}
V_j(z_0) :=\prod_{p\mid P_j(z_0)}\left(1 - \frac{\bm{\omega}_1(m, p)}{p}\right)
\end{equation*}
for $j= 3, 4$. As is standard, we consider $D$ and $\beta$ to be fixed throughout. For $\beta>1$, we define 
\begin{equation}\label{abetarbeta}
a_\beta := e\frac{\beta}{\beta-1}\log \left(\frac{\beta}{\beta-1} \right), \hspace{.4cm} r_\beta := \frac{\log \left(1+\frac{6}{\log 7} \right)}{\log \left(\frac{\beta}{\beta-1} \right)} 
\end{equation}
and 
\begin{equation}\label{Cbetas}
\mathfrak{C}_{\beta}(s):=e^{r_{\beta}-1}\left(1+\frac{6}{\log 7}\right)\frac{a_{\beta}^{\left\lfloor s-\beta\right\rfloor +1}}{1-a_{\beta}}.
\end{equation}

\begin{lemma}\label{lem:lambdadsums}
Let $D>0$ and $\beta \geq 5$ be given and set $s := \frac{\log D}{\log z_0}$. Then for $s \geq \beta$ and $z_0 \geq 7$, we have
\begin{equation*}
\sum_{d \mid P_j(z_0)} \lambda_d^{-} \frac{\bm{\omega}_1(m, d)}{d} > V_j(z_0) (1- \mathfrak{C}_{\beta}(s)) \quad \text{ and } \quad \sum_{d \mid P_j(z_0)} \lambda_d^{+} \frac{\bm{\omega}_1(m, d)}{d} < V_j(z_0) (1+ \mathfrak{C}_{\beta}(s))
\end{equation*}
\end{lemma}
\begin{proof}
For simplicity, we first denote
\[
V_j^-(z_0) := \sum_{d\mid P_j(z_0)} \lambda_{d}^- \frac{\bm{\omega}_1(m, d)}{d}
\hspace{.4cm} \text{and} \hspace{.4cm}
V_j^+(z_0) := \sum_{d\mid P_j(z_0)} \lambda_{d}^+ \frac{\bm{\omega}_1(m, d)}{d}.
\]
Letting 
\[
y_m := \left(\frac{D}{p_1p_2 \cdots p_m} \right)^{1/\beta},
\]
we define 
\begin{equation}\label{DefVjnz0}
V_{j,n}(z_0):=\sum_{\substack{y_n<p_n<\dots p_1<z_0\\ p_m<y_m,\ m<n,\ m\equiv n\Mod 2)}} \frac{\bm{\omega}_1\left(m, p_1p_2\cdots p_n\right)}{p_1p_2\cdots p_n} V_j\left(p_n\right).
\end{equation}
It follows by inclusion-exclusion as in \cite[(6.29) and (6.30)]{IwaniecKowalski} that
\begin{align}\label{Vj-andVj+}
V_j^-(z_0) = V_j(z_0) - \sum_{n \, \text{even}} V_{j,n}(z_0)
\hspace{.4cm}\text{and}\hspace{.4cm}
V_j^+(z_0) = V_j(z_0) + \sum_{n \, \text{odd}} V_{j,n}(z_0).
\end{align}
For $z_n := z_0^{\left(\frac{\beta-1}{\beta} \right)^n}$, we can bound $V_{j,n}(z_0)$ as (cf. \cite[p.~157]{IwaniecKowalski})
\begin{equation}\label{Vjnz0}
V_{j,n}(z_0) \leq \frac{V_j(z_n)}{n!}\left(\log \left(\frac{V_j(z_n)}{V_j(z_0)}\right) \right)^n,
\end{equation}
where $V_j(z_n)$ can be expressed as
\[
V_j(z_n) = \prod_{p\mid P_j(z_n)}\left(1 - \frac{\bm{\omega}_1(m, p)}{p}\right) = V_j(z_0) \prod_{z_n\leq p <z}\left(1 - \frac{\bm{\omega}_1(m, p)}{p}\right)^{-1}.
\] 
Applying Lemma \ref{Boundinlog} into the above equation, we obtain
\begin{equation}\label{Vjzn}
V_j(z_n) \leq V_j(z_0) \left(\frac{\log z_0}{\log z_n} \right) \left( 1+ \frac{6}{\log z_n} \right).
\end{equation}
Note that $z_n \geq 7$. Thus by inserting \eqref{Vjzn} into \eqref{Vjnz0}, $V_{j,n}(z_0)$ can be bounded as
\[
V_{j,n}(z_0) \leq \frac{V_j(z_0)}{n!}  \left(\frac{\log z_0}{\log z_n} \right) \left( 1+ \frac{6}{\log 7} \right) \left[\log\left\{\frac{\log z_0}{\log z_n} \left( 1+ \frac{6}{\log 7} \right)\right\}\right]^n.
\]
It follows from well-known Stirling's bound (a more precise version by Robbins \cite{Robbins}) that
\[
n! \geq \sqrt{2\pi n}\left(\frac{n}{e} \right)^n \geq e \left(\frac{n}{e} \right)^n.
\]
Therefore by utilizing $z_n = z_0^{\left(\frac{\beta-1}{\beta} \right)^n}$, we have
\begin{align*}
V_{j,n}(z_0) &\leq \frac{V_j(z_0)}{e n^n} \left(e\frac{\beta}{\beta-1} \right)^n  \left( 1+ \frac{6}{\log 7} \right) \left[\log\left\{\left(\frac{\beta}{\beta-1} \right)^n  \left( 1+ \frac{6}{\log 7} \right)\right\}\right]^n \\
&= \frac{V_j(z_0)}{e} \left(e\frac{\beta}{\beta-1} \log \left(\frac{\beta}{\beta-1} \right) \right)^n  \left( 1+ \frac{6}{\log 7} \right) \left(1 + \frac{\log \left( 1+ \frac{6}{\log 7} \right)}{n\log \left(\frac{\beta}{\beta-1} \right)}\right)^n \\
&= \frac{V_j(z_0)}{e} a_\beta^n \left( 1+ \frac{6}{\log 7} \right) \left(1+ \frac{r_\beta}{n}  \right)^n \\
&\leq V_j(z_0) a_\beta^n e^{r_\beta -1} \left( 1+ \frac{6}{\log 7} \right), 
\end{align*}
where in the penultimate step we used \eqref{abetarbeta} and in the last step we have applied the bound $\left(1+ \frac{r_\beta}{n}  \right)^n \leq e^{r_\beta}$. We next take the sum over all $n\in \N$ on both the sides of the above equation. The definition \eqref{DefVjnz0} yields that $V_{j,n}(z_0) = 0$ for $n \leq s-\beta$. Therefore for $\beta\geq 5$, we have
\[
\sum_{n\geq 1} V_{j,n}(z_0) \leq V_j(z_0) e^{r_\beta -1} \left( 1+ \frac{6}{\log 7} \right) \sum_{n> s-\beta} a_\beta^n \leq V_j(z_0) \mathfrak{C}_{\beta}(s),
\]
where in the last step we have applied \eqref{Cbetas}. Finally, we insert the above bound into \eqref{Vj-andVj+} to conclude
\[
V_j^-(z_0) > V_j(z_0) (1- \mathfrak{C}_{\beta}(s)) \quad \text{ and } \quad V_j^+(z_0)< V_j(z_0) (1+ \mathfrak{C}_{\beta}(s)).
\]
This completes the proof of the lemma.
\end{proof}
We next bound the sums of the type in Lemma \ref{lem:lambdadsums} under the additional restriction that we only sum over those $d$ with $\delta\mid d$, for some $\delta\in\N$.
\begin{lemma}\label{Lem:Bound for S(A, P, z)}
Let $D>0$ and $\beta\geq 5$ be given and set $s:=\frac{\log(D)}{\log(z)}$.  Then for $s\geq \beta$, $z_0 \geq 7$ and square-free $\delta\in\N$, we have
\begin{equation*}
\sum_{\substack{d \mid P_j(z_0)\\ \delta \mid d}} \lambda_d^{-} \, \frac{\bm{\omega}_1(m, d)}{d} \geq \mu(\delta) \left( \prod_{p \mid \delta} \frac{\bm{\omega}_1(m,p)}{p- \bm{\omega}_1(m,p)}\right) V_j(z_0) (1- \mathfrak{C}_{\beta}(s))
\end{equation*}
and
\begin{equation*}
\sum_{\substack{d \mid P_j(z_0)\\ \delta \mid d}} \lambda_d^{+} \, \frac{\bm{\omega}_1(m, d)}{d} \leq \mu(\delta) \left( \prod_{p \mid \delta} \frac{\bm{\omega}_1(m,p)}{p- \bm{\omega}_1(m,p)}\right) V_j(z_0) (1+ \mathfrak{C}_{\beta}(s)).
\end{equation*}
\end{lemma}
\begin{proof}
We first define two characteristic functions
\[
f_{\delta}(n):=\begin{cases} 1&\text{if }\delta\mid n,\\ 0&\text{otherwise}\end{cases}
\hspace{.4cm} \text{and} \hspace{.4cm}
\widetilde{f}_{\delta}(n):=\begin{cases} 1&\text{if }\gcd(n,\delta)=1,\\ 0&\text{otherwise}\end{cases}.
\]
Thus for a prime $p$, we have $f_{p}(n)=1-\widetilde{f}_{p}(n)$. For $\delta$ being square-free,  
\[
f_{\delta}(n)=\prod_{p\mid \delta}f_{p}(n)= \prod_{p\mid \delta}\left(1-\widetilde{f}_{p}(n)\right)=\sum_{u\mid \delta} \mu(u)\widetilde{f}_u(n). 
\]
Therefore by denoting $\bm{\omega}_u(d) := \widetilde{f}_u(d) \bm{\omega}_1(m, d)$, we can write from the above equation that
\begin{multline}\label{Simplifyrestrictedsum}
\sum_{\substack{d \mid P_j(z_0)\\ \delta \mid d}} \lambda_d^{\pm} \, \frac{\bm{\omega}_1(m, d)}{d} = \sum_{d \mid P_j(z_0)} f_{\delta}(d) \lambda_d^{\pm} \, \frac{\bm{\omega}_1(m, d)}{d} = \sum_{u\mid \delta} \mu(u)\sum_{d \mid P_j(z_0)} \widetilde{f}_u(d) \lambda_d^{\pm} \, \frac{\bm{\omega}_1(m, d)}{d}\\
= \sum_{u\mid \delta} \mu(u)\sum_{d \mid P_j(z_0)} \lambda_d^{\pm} \, \frac{\bm{\omega}_u(d)}{d}.
\end{multline}
Letting $V_{j,u}(z_0) :=\prod\limits_{p\mid P_j(z_0)}\left(1 - \frac{\bm{\omega}_u(p)}{p}\right)$, it follows from Lemma \ref{lem:lambdadsums} that
\begin{align*}
\sum_{d \mid P_j(z_0)} \lambda_d^{-} \frac{\bm{\omega}_u(d)}{d} > V_{j,u}(z_0) (1- \mathfrak{C}_{\beta}(s)) \quad \text{ and } \quad \sum_{d \mid P_j(z_0)} \lambda_d^{+} \frac{\bm{\omega}_u(d)}{d} < V_{j,u}(z_0) (1+ \mathfrak{C}_{\beta}(s)).
\end{align*}
Thus \eqref{Simplifyrestrictedsum} can be derived as
\begin{align}\label{applyprevlemma1}
\sum_{\substack{d \mid P_j(z_0)\\ \delta \mid d}} \lambda_d^{-} \, \frac{\bm{\omega}_1(m, d)}{d} > \sum_{u\mid \delta} \mu(u)V_{j,u}(z_0) (1- \mathfrak{C}_{\beta}(s))
\end{align}
and 
\begin{align}\label{applyprevlemma2}
\sum_{\substack{d \mid P_j(z_0)\\ \delta \mid d}} \lambda_d^{+} \, \frac{\bm{\omega}_1(m, d)}{d} < \sum_{u\mid \delta} \mu(u)V_{j,u}(z_0) (1+ \mathfrak{C}_{\beta}(s))
\end{align}
The sum on the right hand side of both \eqref{applyprevlemma1} and \eqref{applyprevlemma2} evaluates as
\begin{align}\label{muuvjuz0}
\sum_{u\mid \delta} \mu(u) V_{j,u}(z_0) &= \sum_{u\mid \delta} \mu(u) \prod\limits_{p\mid P_j(z_0)}\left(1 - \frac{\widetilde{f}_u(p) \bm{\omega}_1(m, p)}{p}\right) = V_j(z_0)  \sum_{u\mid \delta} \frac{\mu(u)}{\prod\limits_{p\mid u}\left(1 - \frac{\bm{\omega}_1(m, p)}{p}\right)} \nonumber\\
&= V_j(z_0) \prod_{p\mid \delta}\left( 1 - \frac{p}{p - \bm{\omega}_1(m, p)}\right) = V_j(z_0) \prod_{p\mid \delta}\left(- \frac{\bm{\omega}_1(m, p)}{p - \bm{\omega}_1(m, p)}\right).
\end{align}
Therefore by inserting \eqref{muuvjuz0} into \eqref{applyprevlemma1} and \eqref{applyprevlemma2}, we can conclude our lemma.
\end{proof}

We next define two functions
\begin{equation}\label{Sigma-}
\Sigma^{-}(D, z_0) := \sum_{d_3\mid P_3(z_0)} \sum_{d_4\mid P_4(z_0)} \lambda_{d_3}^{-}\lambda_{d_4}^{-} \frac{\bm{\omega}(m, \bm{\mathrm{d}})}{d_3 d_4}
\end{equation}
and
\begin{equation}\label{Sigma+}
\Sigma^{+}(D, z_0) := \sum_{d_3\mid P_3(z_0)} \sum_{d_4\mid P_4(z_0)} \lambda_{d_3}^{+}\lambda_{d_4}^{+} \frac{\bm{\omega}(m, \bm{\mathrm{d}})}{d_3 d_4}.
\end{equation}
In the following lemma we provide an upper and lower bound of $S(\mathcal{A}, \mathscr{P}_7, z_0)$.
\begin{lemma}\label{Lem:Bound of Sapz 2}
For $D>0$, $\beta \geq 6$, we have
\begin{equation*}
\Sigma^{-}(D, z_0) r_4(m) - \sum_{\substack{d_3 \mid P_3(z_0)\\d_3 < \frac{D}{5^{\beta}}}}\sum_{\substack{d_4 \mid P_4(z_0)\\d_4 < \frac{D}{3^{\beta}}}} |\bm{R}(m, \bm{\mathrm{d}})| \leq S(\mathcal{A}, \mathscr{P}_7, z_0) \leq \Sigma^{+}(D, z_0) r_4(m) + \sum_{\substack{d_3 \mid P_3(z_0)\\d_3 < \frac{D}{5^{\beta}}}}\sum_{\substack{d_4 \mid P_4(z_0)\\d_4 < \frac{D}{3^{\beta}}}} |\bm{R}(m, \bm{\mathrm{d}})|
\end{equation*}
\end{lemma}
\begin{proof}
We can write $S(\mathcal{A}, \mathscr{P}_7, z_0)$  as
\begin{align*}
S(\mathcal{A}, \mathscr{P}_7, z_0) = \sum_{\substack{(x_3, x_4) \in \mathcal{A}\\ (x_j, p_j(z_0)) =1}} 1 &= \sum_{(x_3, x_4) \in \mathcal{A}} \left(\sum_{d_3 \mid (x_3, P_3(z_0))} \mu(d_3) \right) \left(\sum_{d_4 \mid (x_4, P_4(z_0))} \mu(d_4) \right)\\
&= \sum_{d_3\mid P_3(z_0)} \sum_{d_4\mid P_4(z_0)} \mu(d_3) \mu(d_4) |\mathcal{A}_{\bm{\mathrm{d}}}|
\end{align*}
Thus the inequality $\lambda_{d}^{-} \leq \mu(d) \leq \lambda_{d}^{+}$ immediately yields 
\begin{equation*}
\sum_{d_3\mid P_3(z_0)}\sum_{d_4\mid P_4(z_0)} \lambda_{d_3}^{-}\lambda_{d_4}^{-} |\mathcal{A}_{\bm{\mathrm{d}}}| \leq S(\mathcal{A}, \mathscr{P}_7, z_0) \leq \sum_{d_3\mid P_3(z_0)}\sum_{d_4\mid P_4(z_0)} \lambda_{d_3}^{+}\lambda_{d_4}^{+} |\mathcal{A}_{\bm{\mathrm{d}}}|.
\end{equation*}
Invoking Proposition \ref{Proposition 1:2} in the above equation, we arrive at
\begin{align*}
\sum_{d_3\mid P_3(z_0)}\sum_{d_4\mid P_4(z_0)} \lambda_{d_3}^{-}\lambda_{d_4}^{-} \left( \frac{\bm{\omega}(m, \bm{\mathrm{d}})}{d_3d_4}r_4(m) + \bm{R}(m, \bm{\mathrm{d}}) \right) &\leq S(\mathcal{A}, \mathscr{P}_7, z_0)\\
& \leq \sum_{d_3\mid P_3(z_0)}\sum_{d_4\mid P_4(z_0)} \lambda_{d_3}^{+}\lambda_{d_4}^{+} \left(\frac{\bm{\omega}(m, \bm{\mathrm{d}})}{d_3d_4}r_4(m) + \bm{R}(m, \bm{\mathrm{d}})\right).
\end{align*}
It follows from the definition of the Rosser weights that $|\lambda_{d_j}^{\pm}|\leq 1$, $\lambda_{d_3}^{\pm} = 0$ for $d_3 \geq \frac{D}{5^{\beta}}$ and $\lambda_{d_4}^{\pm} = 0$ for $d_4 \geq \frac{D}{3^{\beta}}$ where $\beta \geq 6$. Therefore by applying definitions of $\Sigma^{-}(D, z_0)$ and $\Sigma^{+}(D, z_0)$ from \eqref{Sigma-} and \eqref{Sigma+} respectively and and inserting the absolute value termwise for the sum on $\bm{R}(m, \bm{\mathrm{d}})$, we can conclude our lemma.
\end{proof}

\subsection{Bounds for the main term from sieving} 
We define a multiplicative function
\begin{equation*}
g(\eta) := \prod_{p \mid \eta} \bm{\omega}_2(m, p)
\end{equation*} 
and 
\begin{equation*}
\Sigma_{\mathrm{MT}} (D, z_0):=  \sum_{d_{34}\mid P(z_0)}g(d_{34}) \sum_{\ell \mid \frac{P(z_0)}{d_{34}}} \mu(\ell) \prod_{j=3}^4 V_j(z_0) \mu(\xi_j) \prod_{p\mid \xi_j} \frac{\bm{\omega}_1(m, p)}{p - \bm{\omega}_1(m, p)}.
\end{equation*}
In the following lemma we bound $\Sigma^{-}(D, z_0)$ from below to obtain a lower bound for $S(\mathcal{A}, \mathscr{P}_7, z_0)$ from Lemma \ref{Lem:Bound for S(A, P, z)}.
\begin{lemma}\label{Lem:Lower bound of Sigma-}
For $\beta \geq 5$ and $\mathcal{C}_\beta(s) < 1$, we have
\begin{equation*}
\Sigma^{-}(D, z_0) \geq (1- \mathcal{C}_\beta(s))^2 \Sigma_{\mathrm{MT}} (D, z_0).
\end{equation*}
\end{lemma}
\begin{proof}
It follows from the definition of $\bm{\omega}(m, \bm{\mathrm{d}})$ that 
\begin{equation}\label{omega in g}
\bm{\omega}(m, \bm{\mathrm{d}}) = \bm{\omega}_1(m, d_3) \bm{\omega}_1(m, d_4) g(d_{34}).
\end{equation}
Therefore we can write $\Sigma^{-}(D, z_0)$ as
\begin{align}\label{SigmaDz0 bound}
\Sigma^{-}(D, z_0) &= \sum_{d_3\mid P_3(z_0)} \sum_{d_4\mid P_4(z_0)} \lambda_{d_3}^{-}\lambda_{d_4}^{-} \frac{\bm{\omega}_1(m, d_3) \bm{\omega}_1(m, d_4)}{d_3 d_4}g(d_{34})\nonumber\\
&= \sum_{d_{34}\mid P(z_0)} g(d_{34}) \underset{\gcd \left(d_3,d_4\right)=d_{34}}{\sum_{d_3\mid P_3(z_0)} \sum_{d_4\mid P_4(z_0)}} \lambda_{d_3}^{-}\lambda_{d_4}^{-} \frac{\bm{\omega}_1(m, d_3) \bm{\omega}_1(m, d_4)}{d_3 d_4}
\nonumber\\
&= \sum_{d_{34}\mid P(z_0)} g(d_{34}) S^{-}(d_{34})
\end{align}
where
\[
S^{-}(d_{34}) := \underset{\gcd \left(d_3,d_4\right)=d_{34}}{\sum_{d_3\mid P_3(z_0)} \sum_{d_4\mid P_4(z_0)}} \lambda_{d_3}^{-}\lambda_{d_4}^{-} \frac{\bm{\omega}_1(m, d_3) \bm{\omega}_1(m, d_4)}{d_3 d_4}.
\]
We next handle the sum $S^{-}(d_{34})$. Rewriting the condition $\gcd \left(d_3,d_4\right)=d_{34}$, the sum can be written as 
\begin{align*}
S^{-}(d_{34}) &= \left(\sum_{\substack{d_3 \mid P_3(z_0)\\\left(\frac{d_3}{d_{34}}, \frac{d_4}{d_{34}}\right)=1}} \lambda_{d_3}^{-}\frac{\bm{\omega}_1(m, d_3)}{d_3}\right) \left(\sum_{\substack{d_4 \mid P_4(z_0)\\\left(\frac{d_3}{d_{34}}, \frac{d_4}{d_{34}}\right)=1}} \lambda_{d_4}^{-}\frac{\bm{\omega}_1(m, d_4)}{d_4}\right)\\
&= \left(\sum_{d_3 \mid P_3(z_0)}\sum_{\ell \mid \left(\frac{d_3}{d_{34}}, \frac{d_4}{d_{34}}\right)} \mu(\ell) \lambda_{d_3}^{-}\frac{\bm{\omega}_1(m, d_3)}{d_3}\right) \left(\sum_{d_4 \mid P_4(z_0)}\sum_{\ell \mid \left(\frac{d_3}{d_{34}}, \frac{d_4}{d_{34}}\right)} \mu(\ell) \lambda_{d_4}^{-}\frac{\bm{\omega}_1(m, d_4)}{d_4}\right).
\end{align*}
Now setting $\xi = \ell \, d_{34}$, the above equation reduces to
\begin{align*}
S^{-}(d_{34}) = \sum_{\ell \mid \frac{P(z_0)}{d_{34}}} \prod_{j=3}^4 \left( \sum_{\substack{d_j \mid P_j(z_0)\\ \xi \mid d_j}} \lambda_{d_j}^{-}\frac{\bm{\omega}_1(m, d_j)}{d_j} \right)
\end{align*}
We next apply Lemma \ref{Lem:Bound for S(A, P, z)} inside the product of the above equation to bound $S^{-}(d_{34})$ from below as
\begin{align}\label{Bound of Sd34}
S^{-}(d_{34}) \geq \sum_{\ell \mid \frac{P(z_0)}{d_{34}}} \mu(\ell) \prod_{j=3}^4  \left(\mu(\xi_j) \prod_{p\mid \xi_j} \frac{\bm{\omega}_1(m, p)}{p - \bm{\omega}_1(m, p)} V_j(z_0) (1- \mathcal{C}_{\beta}(s))\right).
\end{align}
Finally inserting \eqref{Bound of Sd34} into \eqref{SigmaDz0 bound}, we can conclude our lemma.
\end{proof}
The next lemma provides the upper bound of $\Sigma^{+}(D, z_0)$.
\begin{lemma}\label{Lem:Upper bound of Sigma+}
For $\beta \geq 5$, we have
\begin{equation*}
\Sigma^{+}(D, z_0) \leq (1+ \mathcal{C}_\beta(s))^2 \Sigma_{\mathrm{MT}} (D, z_0).
\end{equation*}
\end{lemma}
\begin{proof}
We can rephrase $\Sigma^{+}(D, z_0)$ from \eqref{omega in g} that 
\begin{equation*}
\Sigma^{+}(D, z_0) = \sum_{d_{34}\mid P(z_0)} g(d_{34}) S^{+}(d_{34})
\end{equation*}
where
\[
S^{+}(d_{34}) := \underset{\gcd \left(d_3,d_4\right)=d_{34}}{\sum_{d_3\mid P_3(z_0)} \sum_{d_4\mid P_4(z_0)}} \lambda_{d_3}^{+}\lambda_{d_4}^{+} \frac{\bm{\omega}_1(m, d_3) \bm{\omega}_1(m, d_4)}{d_3 d_4}.
\]
Now applying Lemma \ref{Lem:Bound for S(A, P, z)}, the proof follows similarly as in the proof of Lemma \ref{Lem:Lower bound of Sigma-}.
\end{proof}
\begin{lemma}
We have
\begin{equation*}
\Sigma_{\mathrm{MT}}(D, z_0) = \left(1-\frac{\bm{\omega}_1(m, 3)}{3}\right) \left(1-\frac{\bm{\omega}_1(m, 5)}{5}\right) \prod_{p\mid P(z_0)} \left(1- \frac{\Omega(m, p)}{p}\right)
\end{equation*}
\end{lemma}
\begin{proof}
It follows from Lemma \ref{Lem:Lower bound of Sigma-} and Lemma \ref{Lem:Upper bound of Sigma+} that 
\[
(1- \mathcal{C}_\beta(s))^2 \Sigma_{\mathrm{MT}} (D, z_0) \leq \Sigma^{-}(D, z_0)\leq \Sigma^{+}(D, z_0) \leq (1+ \mathcal{C}_\beta(s))^2 \Sigma_{\mathrm{MT}} (D, z_0)
\]
Now as $D\to \infty$, we have $\lambda_{d_j}^{-} = \lambda_{d_j}^{+} = \mu(d_j)$. Therefore the definitions \eqref{Sigma-} and \eqref{Sigma+} together implies
\[
\lim_{D\to \infty} \Sigma^{-}(D, z_0) = \lim_{D\to \infty} \Sigma^{+}(D, z_0) = \left(1-\frac{\bm{\omega}_1(m, 3)}{3}\right) \left(1-\frac{\bm{\omega}_1(m, 5)}{5}\right) \prod_{p\mid P(z_0)} \left(1- \frac{\Omega(m, p)}{p}\right).
\]
On the other hand, for $\beta \geq 5$, $a_\beta < 1$ and $s = \frac{\log D}{\log z_0} \to \infty$ as $D \to \infty$, thus $ \mathcal{C}_\beta(s) \to 0$ as $D \to \infty$. This completes the proof of our lemma.
\end{proof}
We next provide the lower bound for $\Sigma_{\mathrm{MT}}(D, z_0)$.
\begin{lemma}\label{Lem: Lower bound of sigmamt}
For $z_0 \geq 7$, we have
\[
\Sigma_{\mathrm{MT}}(D, z_0) \geq \frac{1.39}{(\log z_0)^3}
\]
\end{lemma}
\begin{proof}
The bound of $\bm{\omega}_1(m, p)$ in \eqref{omega1 bound:2} along with Lemma \ref{Big omega bound} yields
\[
\Sigma_{\mathrm{MT}}(D, z_0) \geq \frac{3}{8} \prod_{p \mid P(z_0)} \left(1 - \frac{2.5}{p}\right).
\]
It can be observed that for $p \geq 7$, we have $1 - \frac{2.5}{p} \geq \left(1 - \frac{1}{p} \right)^3$. Thus, we can write
\begin{equation}\label{SigmaMT bound:2}
\Sigma_{\mathrm{MT}}(D, z_0) \geq \frac{3}{8} \prod_{p \mid P(z_0)} \left(1 - \frac{1}{p}\right)^3 \geq 19.77 \prod_{p < z_0} \left(1 - \frac{1}{p}\right)^3.
\end{equation}
For $\gamma$ denoting the Euler's constant, the result \cite[Equation 3.30, p. 70]{Rosser} 
\begin{equation*}\label{Rosser and schoenfield}
\prod_{p < z_0} \left(1 - \frac{1}{p}\right) > \frac{e^{-\gamma}}{\log z_0}\left(1+\frac{1}{\log^2 z_0}\right)^{-1},
\end{equation*}
reduces the bound of $\Sigma_{\mathrm{MT}}(D, z_0)$ in \eqref{SigmaMT bound:2}  as
\[
\Sigma_{\mathrm{MT}}(D, z_0) \geq 19.77 \frac{e^{-3\gamma}}{(\log z_0)^3}\left(1+\frac{1}{\log^2 z_0}\right)^{-3}\geq 19.77 \frac{e^{-3\gamma}}{(\log z_0)^3}\left(1-\frac{1}{\log^2 7}\right)^{3} \geq \frac{1.39}{(\log z_0)^3}
\]
where in the penultimate step we have used the fact that $z_0 \geq 7$. This completes the proof of the lemma.
\end{proof}

\subsection{Bounds for the error term from sieving} 
We next bound the cuspidal contribution to obtain a bound for $S(\mathcal{A}, \mathscr{P}_7, z_0)$.
\begin{lemma}\label{Lem:Bound of Sapz error 2}
For $\beta\geq 7$, we have
\begin{equation*}
\sum_{\substack{d_3 \mid P_3(z_0)\\d_3 < \frac{D}{5^{\beta}}}}\sum_{\substack{d_4 \mid P_4(z_0)\\d_4 < \frac{D}{3^{\beta}}}} |\bm{R}(m, \bm{\mathrm{d}})| \leq 1.38 \times 10^{-34} m^{\frac{3}{5}} D^{14.02}
\end{equation*}
\end{lemma}
\begin{proof}
The proof of the lemma follows immediately from Lemma \ref{Upper bound on Rmd:2} by applying the trivial bound $d_3d_4 <\frac{D^2}{15^7}$ for $\beta\geq 7$. 
\end{proof}
\subsection{The proof of Theorem \ref{Thm 2}}
We next invoke Lemma \ref{Lem:Bound of Sapz error 2} into Lemma \ref{Lem:Bound of Sapz 2} to obtain
\begin{equation*}
\Sigma^{-}(D, z_0) r_4(m) - 1.38 \times 10^{-34} m^{\frac{3}{5}} D^{14.02} \leq S(\mathcal{A}, \mathscr{P}_7, z_0) \leq \Sigma^{+}(D, z_0) r_4(m) + 1.38 \times 10^{-34} m^{\frac{3}{5}} D^{14.02}.
\end{equation*}
The following lemma provides the lower bound of $S(\mathcal{A}, \mathscr{P}_7, z_0)$ for $\beta = 7$ and $D\geq z_0^{21}$.
\begin{lemma}\label{Lem:Lower bound of SApz0}
For $\beta = 7$ and $D\geq z_0^{21}$, we have
\begin{equation*}
S(\mathcal{A}, \mathscr{P}_7, z_0) \geq \frac{0.23\, r_4(m)}{(\log z_0)^{3}} - 1.38 \times 10^{-34} m^{\frac{3}{5}} D^{14.02}
\end{equation*}
\end{lemma}
\begin{proof}
It follows from the definition \eqref{Cbetas} that for $\beta = 7$ and $D\geq z_0^{21}$,
\[
\mathfrak{C}_{\beta}(s) \leq \frac{3}{5}.
\] 
Inserting the above bound and the bound from Lemma \ref{Lem: Lower bound of sigmamt} into Lemma \ref{Lem:Lower bound of Sigma-}, we can bound $\Sigma^{-}(D, z_0)$ from below as
\[
\Sigma^{-}(D, z_0) \geq 0.23 (\log z_0)^{-3}.
\]
This completes the proof of the lemma.
\end{proof}
We are now ready to prove Theorem \ref{Thm 2}. 
\begin{proof}[Proof of Theorem \ref{Thm 2}]
It follows from \eqref{Jacobi} that for $4\nmid m$, one can bound trivially $r_4(m)$ as 
\begin{equation}\label{Lower bound Jacobi}
r_4(m) \geq 8m.
\end{equation}
Therefore Lemma \ref{Lem:Lower bound of SApz0} together with the bound \eqref{Lower bound Jacobi} implies that for $4\nmid m$, the number of solutions to the equation $x^2 + y^2 + (2^a3^b z_1)^2 + (2^c5^d z_2)^2 = m$ with $p\mid z_1, z_2$ as long as $p \geq z_0$, can be written as
\[
S(\mathcal{A}, \mathscr{P}_7, z_0) \geq \frac{1.84\, m}{(\log z_0)^{3}} - 1.38 \times 10^{-34} m^{\frac{3}{5}} D^{14.02},
\]
where $D \geq z_0^{21}$. We next choose $D = z_0^{21}$ and $z_0 = m^{\frac{1}{738}}$ to obtain
\[
S(\mathcal{A}, \mathscr{P}_7, z_0) \geq \frac{1.84 \cdot (738)^3 \, m}{(\log m)^{3}} - 1.38 \times 10^{-34} m^{0.99895}.
\]
Applying the bound 
\[
\log m \leq \frac{1}{r}m^r
\]
for $r = 10^{-6}$, we obtain
\[
S(\mathcal{A}, \mathscr{P}_7, z_0) \geq 7.39 \times 10^{-10} \, m^{0.99999} - 1.38 \times 10^{-34} m^{0.99895}.
\]
Clearly, $S(\mathcal{A}, \mathscr{P}_7, z_0)$ is positive as long as 
\[
m^{0.00104} \geq 1.86 \times 10^{-25},
\]
which holds trivially for any natural number $m$. Therefore, for every $m \in \N$ with $4\nmid m$ we have a representation 
\begin{equation}\label{Representation of m}
m = x^2 + y^2 + (2^a3^bz_1)^2 + (2^c 5^d z_2)^2
\end{equation}
where $x, y, a, b, c, d$ are non-negative integers and $z_1, z_2$ has at most $369$ prime factors. 

Now, for $4\mid m$, we write $m = 4^{\ell} m_0$ such that $\gcd (4, m_0) = 1$. It follows from \eqref{Representation of m} that we can represent $m_0$ as
\[
m_0 = x'^2 + y'^2 + (2^{a'}3^{b'}z_1')^2 + (2^{c'} 5^{d'} z_2')^2
\]
for some non-negative integers $x', y', a', b', c', d'$ and $z_1', z_2'$ with at most $369$ prime factors. Therefore
\[
m = (2^\ell x')^2 + (2^\ell y')^2 + (2^{a'+\ell}3^{b'}z_1')^2 + (2^{c'+\ell} 5^{d'} z_2')^2,
\]
which concludes that every $m$ can be represented in the form of 
\[
m = x^2 + y^2 + (2^a3^bz_1)^2 + (2^c 5^d z_2)^2,
\]
for some non-negative integers $x, y, a, b, c, d$ and $z_1, z_2$ with at most $369$ prime factors. This completes the proof of our theorem.
\end{proof}

\subsection*{Acknowledgements} 
The author would like to show his sincere gratitude to Prof. Ben Kane for fruitful discussions and suggestions on the manuscript. The author is currently a postdoctoral fellow at IIT Gandhinagar, India supported by the SERB-DST CRG grant CRG/2020/002367 of Prof. Atul Dixit. The author sincerely thanks the institute and Prof. Atul Dixit for their support.

\end{document}